\documentclass[11pt,a4paper]{article}

\usepackage[T1]{fontenc}
\usepackage[utf8]{inputenc}
\usepackage{indentfirst}
\usepackage{amsmath,amssymb,amsthm,mathtools}
\usepackage{libertinus} 
\usepackage{microtype}
\usepackage[
  a4paper,
  textwidth=150mm,
  textheight=235mm,
  heightrounded,
  centering
]{geometry}
\usepackage{enumitem}
\usepackage{booktabs}
\usepackage[dvipsnames]{xcolor}
\usepackage{titlesec}
\usepackage{etoolbox}
\usepackage[
  colorlinks=true,
  linkcolor=MidnightBlue,
  citecolor=MidnightBlue,
  urlcolor=MidnightBlue,
  pdftitle={Hilbertian Kahane--Salem--Zygmund Inequalities: Extremizers and Quantitative Gaps},
  pdfauthor={D. Pellegrino and A. Raposo Jr.},
  pdfkeywords={Kahane--Salem--Zygmund inequalities, Hilbert spaces, unimodular multilinear forms, Hadamard matrices, Hurwitz--Radon families, orthogonal tensors}
]{hyperref}

\allowdisplaybreaks
\numberwithin{equation}{section}
\titleformat{\section}[block]
  {\centering\normalfont\large\scshape}
  {\thesection.}{0.65em}{}
\titlespacing*{\section}{0pt}{3.2ex plus .8ex minus .2ex}{1.8ex plus .2ex}

\titleformat{\subsection}[runin]
  {\normalfont\bfseries}
  {\thesubsection.}{0.5em}{}[.]
\titlespacing*{\subsection}{0pt}{2.2ex plus .5ex minus .2ex}{0.75em}

\newtheoremstyle{cleanplain}
  {7pt}{7pt}{\itshape}{}
  {\bfseries}{.}{0.5em}{}
\newtheoremstyle{cleandefinition}
  {7pt}{7pt}{\normalfont}{}
  {\bfseries}{.}{0.5em}{}
\newtheoremstyle{cleanremark}
  {7pt}{7pt}{\normalfont}{}
  {\itshape}{.}{0.5em}{}

\usepackage{aliascnt}
\theoremstyle{cleanplain}
\newtheorem{theorem}{Theorem}[section]
\newaliascnt{acknowledgement}{theorem}

\aliascntresetthe{acknowledgement}

\newaliascnt{algorithm}{theorem}

\aliascntresetthe{algorithm}

\newaliascnt{axiom}{theorem}

\aliascntresetthe{axiom}

\newaliascnt{case}{theorem}

\aliascntresetthe{case}

\newaliascnt{claim}{theorem}

\aliascntresetthe{claim}

\newaliascnt{conclusion}{theorem}

\aliascntresetthe{conclusion}

\newaliascnt{condition}{theorem}

\aliascntresetthe{condition}

\newaliascnt{corollary}{theorem}
\newtheorem{corollary}[corollary]{Corollary}
\aliascntresetthe{corollary}

\newaliascnt{criterion}{theorem}

\aliascntresetthe{criterion}

\newaliascnt{lemma}{theorem}
\newtheorem{lemma}[lemma]{Lemma}
\aliascntresetthe{lemma}

\newaliascnt{notation}{theorem}

\aliascntresetthe{notation}

\newaliascnt{proposition}{theorem}
\newtheorem{proposition}[proposition]{Proposition}
\aliascntresetthe{proposition}

\newaliascnt{solution}{theorem}

\aliascntresetthe{solution}

\newaliascnt{summary}{theorem}

\aliascntresetthe{summary}

\newtheorem{mainresult}{Main Theorem}

\newtheorem*{backgroundtheorem}{Background Theorem}
\theoremstyle{cleandefinition}
\newaliascnt{conjecture}{theorem}

\aliascntresetthe{conjecture}

\newaliascnt{definition}{theorem}
\newtheorem{definition}[definition]{Definition}
\aliascntresetthe{definition}

\newaliascnt{example}{theorem}
\newtheorem{example}[example]{Example}
\aliascntresetthe{example}

\newaliascnt{exercise}{theorem}

\newaliascnt{question}{theorem}

\aliascntresetthe{question}

\theoremstyle{cleanremark}
\newaliascnt{remark}{theorem}
\newtheorem{remark}[remark]{Remark}
\aliascntresetthe{remark}

\usepackage[nameinlink,noabbrev]{cleveref}
\crefname{theorem}{Theorem}{Theorems}
\crefname{proposition}{Proposition}{Propositions}
\crefname{lemma}{Lemma}{Lemmas}
\crefname{corollary}{Corollary}{Corollaries}
\crefname{remark}{Remark}{Remarks}
\crefname{definition}{Definition}{Definitions}
\crefname{example}{Example}{Examples}
\crefname{conjecture}{Conjecture}{Conjectures}
\crefname{mainresult}{Main Theorem}{Main Theorems}
\crefalias{proposition}{proposition}
\crefalias{lemma}{lemma}
\crefalias{corollary}{corollary}

\makeatletter
\newcommand{\@address}{}
\newcommand{\@email}{}
\newcommand{\@subjclass}{}
\newcommand{\@keywords}{}
\newcommand{\address}[1]{\gdef\@address{#1}}
\newcommand{\email}[1]{\gdef\@email{#1}}
\newcommand{\subjclass}[2][]{\gdef\@subjclass{#2}}
\newcommand{\keywords}[1]{\gdef\@keywords{#1}}
\renewcommand{\maketitle}{%
  \begin{center}
    \vspace*{0.4em}
    {\LARGE\scshape \@title\par}
    \vspace{1.25em}
    {\normalsize\scshape \@author\par}
    \vspace{0.55em}
    {\small \@address\par}
    \vspace{0.2em}
    {\small\texttt{\@email}\par}
  \end{center}
  \vspace{1.0em}
}
\renewenvironment{abstract}
  {\begin{center}\begin{minipage}{0.91\textwidth}\small
   \noindent\textsc{Abstract.}\enspace}
  {\end{minipage}\end{center}\vspace{0.65em}}
\newcommand{\makefrontmatterdetails}{%
  \begin{center}
  \begin{minipage}{0.91\textwidth}
  \small\noindent\textsc{Keywords.}\enspace\@keywords\par
  \vspace{0.25em}
  \noindent\textsc{2020 Mathematics Subject Classification.}\enspace\@subjclass
  \end{minipage}
  \end{center}
  \vspace{0.75em}
}
\makeatother

\newcommand{\C}{\mathbb C}
\newcommand{\R}{\mathbb R}
\newcommand{\N}{\mathbb N}
\newcommand{\op}{\mathrm{op}}
\newcommand{\HS}{\mathrm{HS}}
\newcommand{\rank}{\operatorname{rank}}
\newcommand{\tr}{\operatorname{tr}}
\newcommand{\Span}{\operatorname{span}}
\newcommand{\Defect}{\mathcal D}
\newcommand{\Fblock}{\mathcal F}

\begin{document}

\title{Hilbertian Kahane--Salem--Zygmund Inequalities:\\ Extremizers and Quantitative Gaps}
\author{D. Pellegrino \qquad A. Raposo Jr.}
\address{Departamento de Matem\'atica, Universidade Federal da Para\'iba, Jo\~ao Pessoa, Brazil;\\
Coordena\c{c}\~ao do Curso de Matem\'atica Bacharelado, Universidade Federal do Maranh\~ao, S\~ao Lu\'is, Brazil.}
\email{dmpellegrino@gmail.com \qquad anselmo.junior@ufma.br}

\maketitle

\begin{abstract}
We study real multilinear forms with coefficients in $\{-1,1\}$ on
finite-dimensional Hilbert spaces. Every trilinear sign form on
$\ell_2^r\times\ell_2^n\times\ell_2^n$ has norm at least $\sqrt n$.
Writing $K_{r,n}$ for the least norm divided by $\sqrt n$, we prove that
$K_{r,n}=1$ exactly when a Hadamard matrix of order $n$ exists and
$r\le\rho(n)$, where $\rho$ is the Hurwitz--Radon function. If equality
fails, we obtain an explicit gap above $\sqrt n$. We also prove two
asymptotic results. If $1\le m_n\le n$ and
$\limsup r_n/\log_2 n<2$, there are sign forms on
$\ell_2^{r_n}\times\ell_2^{m_n}\times\ell_2^n$ with norm
$(1+o(1))\sqrt n$. In the square case, if
$r\ge2\lceil\log_2(8n)\rceil$, then
$K_{r,n}-1\ge c(1+\log_2 n)^{-4}$. We also prove a fourth-moment estimate
in every fixed multilinear order, characterize equality, and give exact and
asymptotic constructions.
\end{abstract}

\makefrontmatterdetails

\begingroup
\setcounter{tocdepth}{2}
\footnotesize
\tableofcontents
\endgroup
\vspace{0.5em}

\section{Introduction}

The Kahane--Salem--Zygmund inequality gives unimodular tensors with the
correct order of norm growth. Most applications only need this order. Here we
look instead at the sign tensors with smallest possible norm. In the
Hilbertian case, the elementary lower bound already has the same order as the
general upper bound, and equality forces specific identities among the matrix
slices. Our aim is to describe the equality case and to see what can still be
said when the norm is close to the minimum.

All multilinear forms, Hilbert spaces, and coefficient matrices in the main
problem are real. Complex matrices are used only in the Fourier argument in
\Cref{sec:quantitative-obstruction}.

We recall the form of the Kahane--Salem--Zygmund inequality used below.
For $d\ge2$, $p_1,\ldots,p_d\in[1,\infty]$, and positive integers
$n_1,\ldots,n_d$, write $p_j^*$ for the conjugate exponent of $p_j$ and set
\[
 \vartheta(p_1,\ldots,p_d)
 :=\frac{1}{\min_{1\le j\le d}\max\{2,p_j^*\}},
 \qquad
 \alpha(p_j):=\max\left\{\frac12-\frac1{p_j},0\right\}.
\]
There is a form
\[
 A:\ell_{p_1}^{n_1}(\R)\times\cdots\times\ell_{p_d}^{n_d}(\R)
 \longrightarrow\R,
 \qquad
 A(z^{(1)},\ldots,z^{(d)})
 =\sum_{j_1,\ldots,j_d}\varepsilon_{j_1\cdots j_d}
 z_{j_1}^{(1)}\cdots z_{j_d}^{(d)},
\]
with $\varepsilon_{j_1\cdots j_d}\in\{-1,1\}$ and
\begin{equation}\label{eq:general-lp-ksz-intro}
 \|A\|
 \le C_d
 \left(\sum_{j=1}^d n_j\right)^{\vartheta(p_1,\ldots,p_d)}
 \prod_{j=1}^d n_j^{\alpha(p_j)}.
\end{equation}
We use the formulation from \cite{PellegrinoSerranoSilva}; related estimates appear
in \cite{AlbuquerqueRezende}. The probabilistic method goes back to Salem--Zygmund
and Kahane; see \cite{SalemZygmund,Kahane}.

It is useful to compare two cases. In the
all-$\ell_\infty$ case, \eqref{eq:general-lp-ksz-intro} gives, after absorbing
a factor depending only on $d$,
\begin{equation}
 \|A\|
 \le C_d\max_{1\le j\le d}n_j^{1/2}
 \prod_{j=1}^d n_j^{1/2}.
\end{equation}
For every fixed $d$, the constant in this estimate can be made smaller than
$1+\varepsilon$ once all dimensions are sufficiently large
\cite{PellegrinoRaposo}. In the Hilbertian case $p_1=\cdots=p_d=2$, the same
general inequality becomes
\[
 \|A\|\le C_d\left(\sum_{j=1}^d n_j\right)^{1/2},
\]
while fixing all variables except one at coordinate vectors gives the
universal lower bound
\begin{equation}\label{eq:coordinate-lower-bound-intro}
 \|A\|\ge \max_{1\le j\le d}\sqrt{n_j}.
\end{equation}
Thus the order of growth is immediate. The questions here are more precise:
when can equality hold in \eqref{eq:coordinate-lower-bound-intro}, which sign
patterns give equality, and what can be said when the norm is only slightly
larger than the lower bound? The Hilbert-space case is also related to
multivariable von Neumann inequalities; see Varopoulos \cite{Varopoulos}.
Kosiński's work on the three-point Pick problem yields the multivariable von
Neumann inequality for commuting $3\times 3$ matrices; see \cite{Kosinski} and
Knese \cite{Knese} for a detailed account. For Schur multiplication in tensor algebras, see Mantero--Tonge
\cite{ManteroTonge}.

Consider a trilinear form
\[
 A(x,y,z)=\sum_{i=1}^r\sum_{j=1}^m\sum_{k=1}^n
 \varepsilon_{ijk}x_i y_j z_k,
 \qquad \varepsilon_{ijk}\in\{-1,1\},\qquad m\le n,
\]
on $\ell_2^r\times\ell_2^m\times\ell_2^n$, and define
\[
 \Lambda(r,m,n):=\min_A\|A\|,
 \qquad
 K(r,m,n):=\frac{\Lambda(r,m,n)}{\sqrt n},
 \qquad
 K_{r,n}:=K(r,n,n).
\]
The lower bound above says that $K(r,m,n)\ge1$. If
$M_1,\ldots,M_r\in\{-1,1\}^{m\times n}$ are the matrix slices of $A$, then
\begin{equation}
 A(x,y,z)
 =\left\langle\left(\sum_{i=1}^r x_iM_i\right)z,y\right\rangle.
\end{equation}
The problem can therefore be written in terms of
$M(x)=\sum_i x_iM_i$. Equality
$\|A\|=\sqrt n$ is equivalent to
\[
 M(x)M(x)^T=n\|x\|_2^2I_m \qquad (x\in\R^r).
\]
Equivalently, the bilinear map $(x,y)\mapsto M(x)^Ty$ preserves the product
of Euclidean norms up to the factor $\sqrt n$; we call such a map a
\emph{scaled orthogonal multiplication}.

In the square case, two classical objects appear. Each slice must be a
Hadamard matrix, and the relative products of distinct slices form a
Hurwitz--Radon family. If
$n=2^{4a+b}u$, where $u$ is odd and $0\le b\le3$, then
\[
 \rho(n)=8a+2^b,
\]
and $\rho(n)-1$ is the largest cardinality of a family of pairwise
anticommuting skew-symmetric orthogonal real matrices of order $n$. These facts are classical in the theory of orthogonal designs; see, for
example, \cite{GeramitaPullman,Horadam,Seberry,Shapiro}. We prove that they
also give the exact equality condition for the Hilbertian KSZ problem. We
then obtain an explicit gap above the minimum and a lower bound when the
number of slices is large.

\subsection*{Main results}

The first theorem contains the exact equality case, an explicit gap above
the minimum, and two asymptotic statements. The upper construction works in
rectangular format, while the lower bound for large $r$ is proved in the
square case.

\begin{mainresult}[Trilinear Hilbertian KSZ constants]\label{mainA}

\begin{enumerate}[label=\textnormal{(\roman*)},leftmargin=2.35em,itemsep=.55em]
\item \emph{\textbf{Exact equality and a discrete gap.}} Let $r,n\in\N$.
Then
\[
 K_{r,n}=1
 \quad\Longleftrightarrow\quad
 \text{a Hadamard matrix of order $n$ exists and }r\le\rho(n).
\]
If $A$ is a sign form on
$\ell_2^r\times\ell_2^n\times\ell_2^n$ with $\|A\|>\sqrt n$, then
\[
 \frac{\|A\|}{\sqrt n}
 \ge \left(1+\frac{2}{n^3}\right)^{1/4}.
\]

\item \emph{\textbf{Asymptotic construction below} $2\log_2n$\textbf{.}}
Let $0<\delta<2$, and let $(r_n)$ and $(m_n)$ be sequences of
positive integers. If $1\le m_n\le n$ and
\[
 r_n\le(2-\delta)\log_2n
\]
for all sufficiently large $n$, then
\[
 K(r_n,m_n,n)=1+o(1).
\]

\item \emph{\textbf{A lower bound for large} $r$\textbf{.}}
There is a universal constant $c>0$ such that, if
\[
 r\ge2\left\lceil\log_2(8n)\right\rceil,
\]
then
\[
 K_{r,n}\ge1+\frac{c}{(1+\log_2n)^4}.
\]
\end{enumerate}
\end{mainresult}

Part~\emph{(i)} gives the exact equality condition. The integrality of the
matrix entries also gives the explicit gap. Part~\emph{(ii)} shows that the
lower bound can still be approached when $r_n$ grows almost as
$2\log_2 n$. Part~\emph{(iii)} gives a lower bound in the other direction
when $r$ is large enough.

Parts~\emph{(ii)} and~\emph{(iii)} do not give a complete transition. The
lower bound in part~\emph{(iii)} tends to zero with $n$, and the present
methods do not determine whether $K_{r,n}=1+o(1)$ can hold in a larger range. The rectangular case can
behave differently: when $m=1$, equality may hold even for $r=n$.

The proof of part~\emph{(iii)} starts from the fourth-moment estimate. If
$\|A\|\le(1+\varepsilon)\sqrt n$, the matrix slices almost satisfy the
orthogonality and anticommutation relations on average. We select a smaller
family of slices, replace the normalized slices by nearby orthogonal matrices,
and use Hermitian dilations. A finite group and a matrix-valued Fourier
estimate then give the required dimension bound. This leads to the power four
in the final estimate.

The fourth-moment argument also works in higher multilinear order, provided
we keep the first variables separate. Let $k\ge3$, put $q=k-2$, and
write
\[
 A(x^{(1)},\ldots,x^{(q)},y,z)
 =\sum_{i_1,\ldots,i_q,j,s}
 \varepsilon_{i_1\cdots i_qjs}
 x^{(1)}_{i_1}\cdots x^{(q)}_{i_q}y_jz_s.
\]
For $\alpha=(i_1,\ldots,i_q)$, let
$T_\alpha\in\{-1,1\}^{m\times n}$ be the corresponding coefficient matrix
and define
\begin{equation}
 \Phi_A(x^{(1)},\ldots,x^{(q)})
 :=\sum_{i_1,\ldots,i_q}
 x^{(1)}_{i_1}\cdots x^{(q)}_{i_q}T_{i_1,\ldots,i_q}.
\end{equation}
Then
\[
 A(x^{(1)},\ldots,x^{(q)},y,z)
 =\langle\Phi_A(x^{(1)},\ldots,x^{(q)})z,y\rangle.
\]

\begin{mainresult}[A higher-order Hilbertian KSZ estimate]
\label{thm:intro-higher-order-main}
Let
\[
 A:\ell_2^{r_1}\times\cdots\times\ell_2^{r_q}
 \times\ell_2^m\times\ell_2^n\longrightarrow\R,
 \qquad m\le n,
\]
be a sign form. Let $\mu$ be the normalized product measure on
$S^{r_1-1}\times\cdots\times S^{r_q-1}$ and set
\[
 \delta_k(A):=\frac1{mn^2}\int
 \|\Phi_A(x)\Phi_A(x)^T-nI_m\|_{\HS}^2\,d\mu(x).
\]
Then
\[
 \frac{\|A\|}{\sqrt n}\ge\bigl(1+\delta_k(A)\bigr)^{1/4}.
\]
Moreover, the following conditions are equivalent:
\begin{enumerate}[label=\textnormal{(\roman*)},leftmargin=2.2em,itemsep=.2em]
\item $\|A\|=\sqrt n$;
\item $\delta_k(A)=0$;
\item
\[
 \Phi_A(x)\Phi_A(x)^T
 =n\prod_{\nu=1}^q\|x^{(\nu)}\|_2^2I_m
\]
for all $x^{(\nu)}\in\R^{r_\nu}$.
\end{enumerate}
Equivalently, the map
\[
 (x^{(1)},\ldots,x^{(q)},y)
 \longmapsto\Phi_A(x^{(1)},\ldots,x^{(q)})^T y
\]
is a scaled multilinear orthogonal multiplication. Thus, if the norm is close to $\sqrt n$, the averaged $L^2$ error in these
matrix identities is small. The estimate alone does not yield a stability
statement asserting closeness to an exact orthogonal multiplication.
\end{mainresult}

Exact and asymptotic constructions in higher order are given as well.

\subsection*{Relation with previous work}

Without the sign restriction, Li, Nakatsukasa, Soma and Uschmajew
characterized equality in the spectral-to-Frobenius lower bound by orthogonal
tensors \cite[Theorem~3.5]{LiNakatsukasaSomaUschmajew}. The Hadamard and
Hurwitz--Radon results used below are also classical. Our contribution is to
determine exactly when this equality can be realized in the Hilbertian KSZ
problem by a tensor whose canonical coefficients are signs, and then to derive
the explicit gap and the lower bound for large $r$.

Approximate representations in normalized Hilbert--Schmidt norm are studied
by Moore--Russell and Gowers--Hatami \cite{MooreRussell,GowersHatami}. We do
not use their stability theorems. Instead, in
\Cref{sec:quantitative-obstruction} we construct a finite group from
the almost anticommuting matrices and use matrix-valued Fourier analysis to
obtain the dimension bound directly. This argument is related to the
Clifford-algebra dimension arguments of Regev--Vidick \cite{RegevVidick}.
The higher-order result follows from the same fourth-moment calculation,
integrated over a product of spheres.

\medskip
\noindent\textit{Organization.}
\Cref{sec:preliminaries} contains the matrix notation and the
fourth-moment estimate. \Cref{sec:exact-extremizers} proves the exact
square result and the discrete gap. \Cref{sec:quantitative-obstruction}
proves the lower bound for large $r$, and
\Cref{sec:subcritical-construction} gives the asymptotic construction.
\hyperref[thm:intro-higher-order-main]{Main Theorem~B} is proved in \Cref{sec:higher-order-defect}.
\Cref{sec:higher-order-constructions} and Appendix~A contain higher-order
constructions and consequences for best rank-one approximation ratios,
respectively.

\section{Preliminaries and the fourth-moment estimate}\label{sec:preliminaries}

We introduce the coefficient matrices associated with a unimodular
form and recall the facts about Hadamard matrices and Hurwitz--Radon families
used later. For a form $A$, we write $M(x)$ for the corresponding
linear combination of its sign-matrix slices.

All coefficient matrices and multilinear forms below are real. In the Fourier
argument of \Cref{sec:quantitative-obstruction}, we temporarily pass to complex matrix
algebras; this does not change the scalar field of the extremal problem. We
write $I_{d}$ for
the $d\times d$ identity matrix and $\langle\cdot,\cdot\rangle$ for the
Euclidean inner product. For $x\in\R^{d}$, $\| x\|_{2}$ denotes
its Euclidean norm. For $T=(t_{jk})\in\R^{m\times n}$, we use
\[
\| T\|_{\op}:=\sup_{\| z\|_{2}\leq1}\| Tz\|_{2}
\]
and
\[
\| T\|_{\HS}:=\left(  \sum_{j=1}^{m}\sum_{k=1}^{n}|t_{jk}
|^{2}\right)  ^{1/2}=\left(  \tr(TT^{T})\right)  ^{1/2}.
\]
The unsubscripted notation $\| A\|$ is reserved for the norm of a
multilinear form. All asymptotic statements involving $\log_{2}n$ are
understood for integers $n\geq2$.

\subsection{Coefficient operators associated with a unimodular form}

\begin{definition}
Let $r,m,n\in\N$. A tensor
\[
\mathcal{E}=(\varepsilon_{ijk})\in\{-1,1\}^{r\times m\times n}
\]
is called a \emph{sign tensor}. Its associated trilinear form is
\[
A_{\mathcal{E}}(x,y,z)=\sum_{i=1}^{r}\sum_{j=1}^{m}\sum_{k=1}^{n}
\varepsilon_{ijk}x_{i}y_{j}z_{k}.
\]
We also call $A_{\mathcal{E}}$ a real \emph{unimodular trilinear form}. The
spectral norm of $\mathcal{E}$ is
\[
\|\mathcal{E}\|_{\sigma}:=\sup\left\{  |A_{\mathcal{E}}
(x,y,z)|:\left\| x\right\| _{2}\leq1,\ \left\| y\right\| _{2}
\leq1,\ \left\| z\right\| _{2}\leq1\right\}  .
\]
Thus $\|\mathcal{E}\|_{\sigma}=\| A_{\mathcal{E}}\|$. For
$1\leq m\leq n$, define
\[
\Lambda(r,m,n)=\min\left\{  \left\| \mathcal{E}\right\| _{\sigma
}:\mathcal{E}\in\{-1,1\}^{r\times m\times n}\right\}  .
\]
We also use the normalized constants
\[
K(r,m,n):=\frac{\Lambda(r,m,n)}{\sqrt n},
\qquad
K_{r,n}:=K(r,n,n).
\]
Thus $K(r,m,n)\ge1$ whenever $m\le n$.

\end{definition}

The minimum exists because there are only finitely many sign tensors of a
fixed size. Given $\mathcal{E}$ and its associated form $A=A_{\mathcal{E}}$,
for each $i=1,\ldots,r$, define the $m\times n$ matrix slice
\[
M_{i}=(\varepsilon_{ijk})_{\substack{1\leq j\leq m\\1\leq k\leq n}},
\]
and, for $x=(x_{1},\ldots,x_{r})\in\R^{r}$, put
\[
M(x)=\sum_{i=1}^{r}x_{i}M_{i}.
\]
Thus $M(x)$ is a linear operator from $\R^{n}$ to $\R^{m}$.
Expanding the Euclidean inner product gives
\begin{align*}
\langle M(x)z,y\rangle &  =\sum_{j=1}^{m}\left(  \sum_{i=1}^{r}\sum_{k=1}
^{n}x_{i}\varepsilon_{ijk}z_{k}\right)  y_{j}\\
&  =\sum_{i=1}^{r}\sum_{j=1}^{m}\sum_{k=1}^{n}\varepsilon_{ijk}x_{i}y_{j}
z_{k}\\
&  =A(x,y,z).
\end{align*}
Consequently,
\begin{align*}
\| A\| &  =\sup\left\{  \left| \langle M(x)z,y\rangle\right|
:\left\| x\right\| _{2}\leq1,\ \left\| y\right\| _{2}
\leq1,\ \left\| z\right\| _{2}\leq1\right\} \\
&  =\sup_{\left\| x\right\| _{2},\left\| z\right\| _{2}\leq1}
\sup_{\left\| y\right\| _{2}\leq1}|\langle M(x)z,y\rangle|\\
&  =\sup_{\left\| x\right\| _{2},\left\| z\right\| _{2}\leq1}\|
M(x)z\|_{2}\\
&  =\sup_{\| x\|_{2}\leq1}\| M(x)\|_{\op}.
\end{align*}
Here we used
\[
\sup_{\| y\|_{2}\leq1}|\langle v,y\rangle|=\| v\|_{2}
\]
and the definition of the operator norm. Since $M(tx)=tM(x)$, the supremum in
the $x$-variable may be taken over the unit sphere. Therefore
\begin{equation}
\|A\|=\sup_{\|x\|_2=1}\|M(x)\|_{\op}. \label{eq:trilinear-norm-representation}
\end{equation}

\subsection{Exact orthogonal matrix systems}

A Hadamard matrix of order $n$ is a matrix $H\in\{-1,1\}^{n\times n}$
satisfying
\[
HH^{T}=nI_{n}.
\]
Thus the rows, and equivalently the columns, are pairwise orthogonal and have
Euclidean norm $\sqrt n$. An integer $n$ for which such a matrix exists will
be called a Hadamard order. Hadamard matrices exist in every order $2^{a}$,
and also in order $12$. If $H$ and $K$ are Hadamard matrices of orders $s$ and
$t$, respectively, then
\[
(H\otimes K)(H\otimes K)^{T} =(HH^{T})\otimes(KK^{T})=stI_{st},
\]
so $H\otimes K$ is a Hadamard matrix of order $st$; see, for example,
\cite{Horadam,Seberry}.

We also use the Hurwitz--Radon function. Its domain and range are
\[
\rho:\N\longrightarrow\N.
\]
For $n\in\N$, let $\nu_{2}(n)$ be the largest nonnegative integer $v$
such that $2^{v}$ divides $n$. Equivalently,
\[
n=2^{\nu_{2}(n)}u \qquad\text{with $u$ odd}.
\]
Write the integer $\nu_{2}(n)$ uniquely in the form
\[
\nu_{2}(n)=4a+b, \qquad a\in\N\cup\{0\},\qquad0\le b\le3.
\]
Then
\[
n=2^{4a+b}u, \qquad u\text{ odd},
\]
and the Hurwitz--Radon function is defined by
\begin{equation}
\rho(n)=8a+2^{b}. 
\end{equation}

The following elementary consequence will be used repeatedly:
\begin{equation}
\rho(n)\ge2\nu_{2}(n). \label{eq:rho-valuation-bound}
\end{equation}
Indeed, if $\nu_{2}(n)=4a+b$, then
\[
\rho(n)-2\nu_{2}(n) =(8a+2^{b})-2(4a+b)=2^{b}-2b.
\]
Since $2^b\ge 2b$ for $0\le b\le3$, this proves \eqref{eq:rho-valuation-bound}.

\begin{definition}
A family $E_{1},\ldots,E_{s}$ of $n\times n$ real matrices is a
\emph{Hurwitz--Radon family} if
\[
E_{i}^{T}=-E_{i}, \qquad E_{i}^{T}E_{i}=I_{n}, \qquad E_{i}E_{j}=-E_{j}
E_{i}\quad(i\ne j).
\]

\end{definition}

The first relation says that each $E_{i}$ is skew-symmetric, the second says
that it is orthogonal, and the third says that distinct matrices anticommute.
The first two relations also imply
\[
E_{i}^{2}=-I_{n}.
\]

The classical Hurwitz--Radon theorem says that the maximum possible number of
matrices in such a real family is
\[
\rho(n)-1.
\]
In other words,
\[
\max\bigl\{s:\text{there is a Hurwitz--Radon family of $s$ real matrices of
order $n$}\bigr\}=\rho(n)-1;
\]
see \cite{Hurwitz,Radon,Shapiro}.

For our construction, it is important that the maximum can be attained by
matrices with integer entries. Here an \emph{integral matrix} means simply a
matrix all of whose entries belong to $\mathbb{Z}$.

\begin{backgroundtheorem}[Geramita--Pullman \cite{GeramitaPullman}]
For every $n\in\N$, there exists a Hurwitz--Radon family of $\rho(n)-1$ integral matrices of order $n$.
No Hurwitz--Radon family of real, and hence no such family of integral,
matrices of order $n$ can contain more than $\rho(n)-1$ matrices.
\end{backgroundtheorem}

Thus Geramita and Pullman show that the maximum supplied by the real
Hurwitz--Radon theorem is already attained inside $M_{n}(\mathbb{Z})$. We use
only this existence statement.

One concrete consequence will be useful later. If $E\in M_{n}(\mathbb{Z})$ is
orthogonal, then $E$ is a signed permutation matrix. To see this, let $v$ be
one of its columns. Since $v$ has integer entries and $\|v\|_{2}=1$, exactly
one entry of $v$ is equal to $1$ or $-1$, and all the other entries are zero.
The orthogonality of distinct columns forces these nonzero entries to occur in
different rows. Hence every row and every column of $E$ has exactly one
nonzero entry, equal to $1$ or $-1$.

Consequently, if $H$ is a Hadamard matrix and $E$ is one of the integral
orthogonal matrices above, then $HE$ is again a sign matrix: multiplication on
the right by $E$ merely permutes the columns of $H$ and changes some column
signs. This is why the integral version of the theorem is used in \Cref{sec:exact-extremizers}.

\subsection{The fourth-moment estimate}\label{sec:fourth-moment}

The following fourth-moment estimate for the coefficient matrices gives the
universal lower bound, characterizes equality, and is also used in
\Cref{sec:quantitative-obstruction}.

\begin{proposition}[Universal Hilbertian lower bound]
For all positive integers $r$ and $1\le m\le n$,
\[
\Lambda(r,m,n)\ge\sqrt n.
\]

\end{proposition}

\begin{proof}
Let $A$ have associated matrices $M_{1},\ldots,M_{r}$. Fix $i$, and let
$s_{1},\ldots,s_{q}$ be the nonzero singular values of $M_{i}$, where
$q=\rank(M_{i})\le m$. Then
\[
\max_{1\le \nu \le q}s_{\nu} = \|M_{i}\|_{\op}.
\]
Since $M_{i}$ has $mn$ entries of modulus
one,
\[
\|M_{i}\|_{\HS}^{2}=mn.
\]
On the other hand,
\[
\|M_{i}\|_{\HS}^{2} =\sum_{\nu=1}^{q} s_{\nu}^{2} \le q\max_{1\le
\nu\le q}s_{\nu}^{2} = q\|M_{i}\|_{\op}^{2} \le m\|M_{i}\|_{\op}^{2}.
\]
Therefore
\[
\|M_{i}\|_{\op}^{2}\ge\frac{\|M_{i}\|_{\HS}^{2}}{m}=n.
\]
Taking $x=e_{i}$ in \eqref{eq:trilinear-norm-representation}, we obtain
\[
\|A\|\ge\|M(e_{i})\|_{\op}=\|M_{i}\|_{\op}\ge\sqrt n.
\]
Minimizing over all sign tensors gives the stated bound.
\end{proof}

Before proving the next estimate, we record the spherical averages that will
be used. Let
\[
S^{r-1}=\{x\in\R^{r}:\|x\|_{2}=1\}.
\]
Let $\sigma$ be the normalized surface measure on $S^{r-1}$. For
$r=1$, we use the uniform probability measure on
$S^0=\{-1,1\}$.

The measure $\sigma$ is invariant under every orthogonal transformation. In
particular, changing the sign of one coordinate does not change the measure.
It follows immediately that the integral of a monomial is zero whenever one of
its coordinates occurs with an odd exponent.

\begin{lemma}
\label{lem:spherical-moments} For every $1\le i\le r$,
\[
\int_{S^{r-1}}x_{i}^{4}\,d\sigma(x)=\frac{3}{r(r+2)}.
\]
For every pair of distinct indices $1\le i\ne j\le r$,
\[
\int_{S^{r-1}}x_{i}^{2}x_{j}^{2}\,d\sigma(x)=\frac{1}{r(r+2)}.
\]
Moreover, every monomial containing an odd power of some coordinate has
integral zero.
\end{lemma}

\begin{proof}
The assertion about odd powers follows from the reflection
\[
(x_{1},\ldots,x_{i},\ldots,x_{r}) \longmapsto(x_{1},\ldots,-x_{i},\ldots
,x_{r}),
\]
which preserves $S^{r-1}$ and $\sigma$ but changes the sign of the monomial.

For $r=1$, only the first displayed formula is relevant, and it is
immediate. Assume $r\ge2$. By symmetry, the values
\[
a:=\int_{S^{r-1}}x_{1}^{4}\,d\sigma(x), \qquad b:=\int_{S^{r-1}}x_{1}^{2}
x_{2}^{2}\,d\sigma(x)
\]
do not depend on the choice of distinct coordinates. Since $\sum_{i=1}^{r}
x_{i}^{2}=1$ on the sphere,
\begin{align*}
1  &  =\int_{S^{r-1}}\left(  \sum_{i=1}^{r} x_{i}^{2}\right)  ^{2}d\sigma(x)\\
&  =r a+r(r-1)b.
\end{align*}

The random variables $x_{1}$ and $(x_{1}+x_{2})/\sqrt2$ have the same
distribution with respect to $\sigma$, because the latter is obtained from the
former by an orthogonal change of coordinates. Therefore
\begin{align*}
a  &  =\int_{S^{r-1}}\left(  \frac{x_{1}+x_{2}}{\sqrt2}\right)  ^{4}
d\sigma(x)\\
&  =\frac14\int_{S^{r-1}} \bigl(x_{1}^{4}+4x_{1}^{3}x_{2}+6x_{1}^{2}x_{2}
^{2}+4x_{1}x_{2}^{3}+x_{2}^{4}\bigr)d\sigma(x)\\
&  =\frac14(2a+6b),
\end{align*}
where the two terms containing odd powers integrate to zero. Hence $a=3b$.
Substituting this into $1=ra+r(r-1)b$ gives
\[
r(3b)+r(r-1)b=r(r+2)b=1.
\]
Thus
\[
b=\frac1{r(r+2)}, \qquad a=\frac3{r(r+2)}.
\]

\end{proof}

For $i=1,\ldots,r$ and $1\le i<j\le r$, define
\begin{equation}
C_{i}=M_{i}M_{i}^{T}-nI_{m}, \label{eq:diagonal-defect}
\end{equation}
and
\begin{equation}
D_{ij}=M_{i}M_{j}^{T}+M_{j}M_{i}^{T}. \label{eq:pair-defect}
\end{equation}
Set
\begin{equation}
\mathcal{D}(M_{1},\ldots,M_{r}) =2\sum_{i=1}^{r}\|C_{i}\|_{\HS}^{2}
+\left\|  \sum_{i=1}^{r}C_{i}\right\|  _{\HS}^{2} +\sum_{1\le i<j\le
r}\|D_{ij}\|_{\HS}^{2}. \label{eq:defect}
\end{equation}

\begin{theorem}[Fourth-moment estimate]\label{thm:defect-lower}
Let $r,m,n$ be positive integers with $m\le n$,
and let $A$ be a unimodular trilinear form on $\ell_{2}^{r}\times\ell_{2}
^{m}\times\ell_{2}^{n}$. If $M_{1},\ldots,M_{r}$ are its associated matrices,
then
\begin{equation}
\|A\|^{4}\ge n^{2}+ \frac{\mathcal{D}(M_{1},\ldots,M_{r})}{mr(r+2)}.
\label{eq:fourth-moment-defect}
\end{equation}
In particular, $\|A\|=\sqrt n$ if and only if
\begin{equation}
M_{i}M_{j}^{T}+M_{j}M_{i}^{T}=2n\delta_{ij}I_{m} \qquad(1\le i,j\le r),
\label{eq:exact-matrix-relations}
\end{equation}
where $\delta_{ij}$ is the Kronecker delta.
\end{theorem}

\begin{proof}
Let
\[
S^{r-1}=\{x\in\R^{r}:\|x\|_{2}=1\}
\]
and, for $x\in S^{r-1}$, put
\[
T(x)=M(x)M(x)^{T}.
\]
Using \eqref{eq:diagonal-defect} and \eqref{eq:pair-defect}, we have
\[
T(x)=nI_{m}+R(x),
\]
where
\[
R(x)=\sum_{i=1}^{r} x_{i}^{2}C_{i} +\sum_{1\le i<j\le r}x_{i}x_{j}D_{ij}.
\]

Let $\lambda_{1}(x),\ldots,\lambda_{m}(x)\ge0$ be the eigenvalues of $T(x)$.
Since
\[
\|M(x)\|_{\op}^{2}=\lambda_{\max}(T(x)),
\]
we obtain
\begin{align*}
\|M(x)\|_{\op}^{4}  &  =\lambda_{\max}(T(x))^{2}\\
&  \ge\frac1m\sum_{\nu=1}^{m}\lambda_{\nu}(x)^{2}\\
&  =\frac1m\tr(T(x)^{2}).
\end{align*}
Since $\sigma$ is a probability measure, the supremum of a function is at
least its average. Therefore
\begin{align*}
\|A\|^{4}  &  =\sup_{x\in S^{r-1}}\|M(x)\|_{\op}^{4}\\
&  \ge\frac1m\int_{S^{r-1}}\tr(T(x)^{2})\,d\sigma(x).
\end{align*}
Because $R(x)$ is symmetric, $\tr(R(x)^{2}
)=\|R(x)\|_{\HS}^{2}$, and hence
\begin{align*}
\|A\|^{4}  &  \ge\frac1m\int_{S^{r-1}} \tr\bigl((nI_{m}
+R(x))^{2}\bigr)\,d\sigma(x)\\
&  =n^{2}+ \frac{2n}{m}\int_{S^{r-1}}\tr R(x)\,d\sigma(x)
+\frac1m\int_{S^{r-1}}\|R(x)\|_{\HS}^{2}\,d\sigma(x).
\end{align*}
Now
\[
\tr C_{i}=\tr(M_{i}M_{i}^{T})-nm=\|M_{i}
\|_{\HS}^{2}-nm=0.
\]
Moreover, $\int_{S^{r-1}}x_{i}x_{j}\,d\sigma(x)=0$ for $i\ne j$. Therefore
\[
\int_{S^{r-1}}\tr R(x)\,d\sigma(x)=0.
\]

To finish the estimate, expand the last integral. With $\langle X,Y\rangle_{\HS
}=\tr(XY^{T})$, one has
\begin{align*}
\|R(x)\|_{\HS}^{2}  &  =\left\|  \sum_{i} x_{i}^{2}C_{i}+\sum
_{i<j}x_{i}x_{j}D_{ij}\right\|  _{\HS}^{2}\\
&  =\sum_{i} x_{i}^{4}\|C_{i}\|_{\HS}^{2} +2\sum_{i<j}x_{i}^{2}
x_{j}^{2}\langle C_{i},C_{j}\rangle_{\HS}\\
&  \quad+\sum_{i<j}x_{i}^{2}x_{j}^{2}\|D_{ij}\|_{\HS}^{2}
+\mathcal{R}(x),
\end{align*}
where every monomial occurring in $\mathcal{R}(x)$ contains an odd power of at
least one coordinate. Indeed, the mixed products between a term $x_{i}
^{2}C_{i}$ and a term $x_{j}x_{k}D_{jk}$ contain either $x_{i}^{3}x_{k}$ or
$x_{i}^{2}x_{j}x_{k}$, and products of two distinct terms $x_{i}x_{j}D_{ij}$
and $x_{k}x_{\ell}D_{k\ell}$ also contain an odd power unless the two pairs
are the same. Hence \autoref{lem:spherical-moments} gives
\begin{align*}
\int_{S^{r-1}}\|R(x)\|_{\HS}^{2}\,d\sigma(x)  &  =\frac{3}{r(r+2)}
\sum_{i=1}^{r}\|C_{i}\|_{\HS}^{2}\\
&  \quad+\frac{2}{r(r+2)} \sum_{1\le i<j\le r}\langle C_{i},C_{j}
\rangle_{\HS}\\
&  \quad+\frac{1}{r(r+2)} \sum_{1\le i<j\le r}\|D_{ij}\|_{\HS}^{2}.
\end{align*}
Since
\[
\left\|  \sum_{i=1}^{r}C_{i}\right\|  _{\HS}^{2} =\sum_{i=1}
^{r}\|C_{i}\|_{\HS}^{2} +2\sum_{1\le i<j\le r}\langle C_{i}
,C_{j}\rangle_{\HS},
\]
the preceding expression equals
\[
\frac{\mathcal{D}(M_{1},\ldots,M_{r})}{r(r+2)}.
\]
Substitution gives \eqref{eq:fourth-moment-defect}.

Suppose first that $\|A\|=\sqrt n$. Then \eqref{eq:fourth-moment-defect} implies $\mathcal{D}
(M_{1},\ldots,M_{r})=0$. Every term in \eqref{eq:defect} is nonnegative, so $C_{i}=0$
for every $i$ and $D_{ij}=0$ for every $i<j$. These are exactly the relations \eqref{eq:exact-matrix-relations}.

Conversely, assume \eqref{eq:exact-matrix-relations}. Then, for every $x\in\R^{r}$,
\begin{align*}
M(x)M(x)^{T}  &  =\sum_{i,j=1}^{r} x_{i}x_{j}M_{i}M_{j}^{T}\\
&  =\sum_{i=1}^{r} x_{i}^{2}M_{i}M_{i}^{T} +\sum_{1\le i<j\le r}x_{i}
x_{j}(M_{i}M_{j}^{T}+M_{j}M_{i}^{T})\\
&  =n\sum_{i=1}^{r} x_{i}^{2} I_{m}\\
&  =n\|x\|_{2}^{2}I_{m}.
\end{align*}
Hence $\|M(x)\|_{\op}=\sqrt n\,\|x\|_{2}$, and \eqref{eq:trilinear-norm-representation} gives
$\|A\|=\sqrt n$.
\end{proof}

\begin{remark}
The relations \eqref{eq:exact-matrix-relations} are equivalent to
\[
M(x)M(x)^{T}=n\|x\|_{2}^{2}I_{m} \qquad(x\in\R^{r}).
\]
Thus, for $x\ne0$, all $m$ singular values of $M(x)$ are equal to $\sqrt
n\,\|x\|_{2}$. If these singular values are denoted by $s_{1}(x),\ldots
,s_{m}(x)$, then
\[
\tr\bigl((M(x)M(x)^{T})^{2}\bigr)=\sum_{\nu=1}^{m} s_{\nu
}(x)^{4}.
\]
For a fixed value of $\sum_{\nu=1}^{m} s_{\nu}(x)^{2}$, this fourth-power sum
is minimized precisely when all the singular values are equal. This is the
reason for using the fourth power in \Cref{thm:defect-lower}.
\end{remark}

\begin{corollary}[An integral gap]\label{cor:integral-defect-gap}
If no family of sign matrices $M_{1},\ldots,M_{r}
\in\{-1,1\}^{m\times n}$ satisfies \eqref{eq:exact-matrix-relations}, then
\[
\Lambda(r,m,n)\ge\left(  n^{2}+\frac{1}{mr(r+2)}\right)  ^{1/4}>\sqrt n.
\]

\end{corollary}

\begin{proof}
Every term in \eqref{eq:defect} is the square of the Hilbert--Schmidt norm of an
integral matrix. Hence $\mathcal{D}(M_{1},\ldots,M_{r})$ is a nonnegative
integer. Under the hypothesis it cannot vanish, so it is at least $1$. Apply
\Cref{thm:defect-lower}.
\end{proof}

\begin{remark}
The lower bound in \Cref{cor:integral-defect-gap} is small in general, and
we do not claim that it is optimal. We only use it as an explicit separation
from $\sqrt n$.
\end{remark}

\section{Proof of Main Theorem A(i): exact extremizers and a discrete gap}\label{sec:exact-extremizers}

In the square case, each slice is a Hadamard matrix and the relative products
form a Hurwitz--Radon family. This yields the exact classification, while
integrality gives a gap above the minimum.

\begin{theorem}[Exact square extremizers]
\label{thm:exact-square} Let $r,n\in\N$.
The following are equivalent:

\begin{enumerate}[label=\textnormal{(\roman*)},leftmargin=2.2em,itemsep=.35em]
\item $\Lambda(r,n,n)=\sqrt n$;

\item there exist sign matrices $M_{1},\ldots,M_{r}\in\{-1,1\}^{n\times n}$
satisfying
\[
M_{i}M_{j}^{T}+M_{j}M_{i}^{T}=2n\delta_{ij}I_{n};
\]

\item a Hadamard matrix of order $n$ exists and $r\le\rho(n)$.
\end{enumerate}
\end{theorem}

\begin{proof}
The equivalence of (i) and (ii) follows from \Cref{thm:defect-lower}. Assume (ii) and put
\[
U_{i}=n^{-1/2}M_{i}.
\]
Then every $U_{i}$ is orthogonal, and for $i\ne j$,
\begin{equation}
U_{i}U_{j}^{T}+U_{j}U_{i}^{T}=0. \label{eq:normalized-cross-relations}
\end{equation}
In particular, every $M_{i}$ is a Hadamard matrix.

Fix $i\ne j$. Multiplying \eqref{eq:normalized-cross-relations} on the left by $U_{i}^{T}$ and on the
right by $U_{j}$ gives
\[
I_{n}+(U_{i}^{T}U_{j})^{2}=0.
\]
Since $U_{i}^{T}U_{j}$ is orthogonal,
\[
(U_{i}^{T}U_{j})^{T}=(U_{i}^{T}U_{j})^{-1}=-(U_{i}^{T}U_{j}),
\]
so $U_{i}^{T}U_{j}$ is skew-symmetric.

For $i=2,\ldots,r$, define
\[
E_{i}=U_{1}^{T}U_{i}.
\]
Each $E_{i}$ is orthogonal and skew-symmetric. Moreover, for $2\le i\ne j\le
r$,
\[
E_{i}^{T}E_{j}+E_{j}^{T}E_{i} =U_{i}^{T}U_{j}+U_{j}^{T}U_{i}=0.
\]
Since $E_{i}^{T}=-E_{i}$, this identity is equivalent to $E_{i}E_{j}
=-E_{j}E_{i}$. Thus $E_{2},\ldots,E_{r}$ form a Hurwitz--Radon family of
cardinality $r-1$. Hence
\[
r-1\le\rho(n)-1,
\]
so $r\le\rho(n)$.

Conversely, assume (iii). Let $H$ be a Hadamard matrix of order $n$. The
Geramita--Pullman theorem supplies $\rho(n)-1$ integral orthogonal,
skew-symmetric, pairwise anticommuting matrices of order $n$. Since $r\le
\rho(n)$, we may select $r-1$ of them and denote them by $E_{2},\ldots,E_{r}$.
Put $E_{1}=I_{n}$ and
\[
M_{i}=HE_{i}, \qquad i=1,\ldots,r.
\]
Each $E_{i}$ is a signed permutation matrix. Consequently, $HE_{i}$ is
obtained from $H$ by permuting columns and changing some column signs, so
every $M_{i}$ is a sign matrix. The Hurwitz--Radon relations give
\[
E_{i}E_{j}^{T}+E_{j}E_{i}^{T}=2\delta_{ij}I_{n}.
\]
Therefore
\[
M_{i}M_{j}^{T}+M_{j}M_{i}^{T} =H(E_{i}E_{j}^{T}+E_{j}E_{i}^{T})H^{T}
=2n\delta_{ij}I_{n}.
\]
Thus (ii) holds.
\end{proof}

\subsection{An explicit gap above the minimum}

The general estimate in \Cref{sec:fourth-moment} gives a strict lower bound whenever the
relations \eqref{eq:exact-matrix-relations} fail. In the square case, a direct argument using one or
two of the matrices $M_{i}$ gives a stronger separation that is independent of
$r$. In this subsection, $m=n$, and $C_{i},D_{ij}$ are the matrices defined in
\eqref{eq:diagonal-defect} and \eqref{eq:pair-defect}. For distinct indices $i,j$, let $L_{ij}$ denote the norm of the restriction of
$A$ to $\Span\{e_i,e_j\}\times\ell_2^n\times\ell_2^n$.

\begin{lemma}
\label{lem:square-gap-estimates} Let $A$ be a unimodular trilinear form on
$\ell_{2}^{r}\times\ell_{2}^{n}\times\ell_{2}^{n}$, with associated matrices
$M_{1},\ldots,M_{r}$.

\begin{enumerate}
[label=\textup{(\alph*)}]

\item For every $i$,
\[
\|M_{i}\|_{\op}^{4} \ge n^{2}+\frac{\|C_{i}\|_{\HS}^{2}}{n}.
\]

\item If $C_{i}=C_{j}=0$ for two distinct indices $i,j$, then
\[
L_{ij}^{4} \ge n^{2}+\frac{\|D_{ij}\|_{\HS}^{2}}{4n}.
\]

\end{enumerate}
\end{lemma}

\begin{proof}
For (a), put
\[
T=M_{i}M_{i}^{T}=nI_{n}+C_{i}.
\]
If $\lambda_{1},\ldots,\lambda_{n}$ are the eigenvalues of the positive
semidefinite matrix $T$, then
\[
\|M_{i}\|_{\op}^{4} =\bigl(\max_{\nu}\lambda_{\nu}\bigr)^{2} \ge
\frac1n\sum_{\nu=1}^{n}\lambda_{\nu}^{2} =\frac1n\tr(T^{2}).
\]
Since
\[
\tr C_{i}=\|M_{i}\|_{\HS}^{2}-n^{2}=0,
\]
we obtain
\[
\frac1n\tr(T^{2}) =n^{2}+\frac{\|C_{i}\|_{\HS}^{2}}{n}.
\]

For (b), consider the two unit vectors
\[
x_{+}=\frac{e_{i}+e_{j}}{\sqrt2}, \qquad x_{-}=\frac{e_{i}-e_{j}}{\sqrt2}.
\]
The assumption $C_{i}=C_{j}=0$ means $M_{i}M_{i}^{T}=M_{j}M_{j}^{T}=nI_{n}$.
Therefore
\[
M(x_{+})M(x_{+})^{T}=nI_{n}+\frac12D_{ij}, \qquad M(x_{-})M(x_{-})^{T}
=nI_{n}-\frac12D_{ij}.
\]
Applying the same eigenvalue estimate to each of these two matrices gives
\[
L_{ij}^{4}\ge\frac1n\tr\left(  nI_{n}\pm\frac12D_{ij}\right)
^{2}.
\]
Since $L_{ij}^{4}$ is at least both quantities, it is at least their average.
Thus
\begin{align*}
L_{ij}^{4}  &  \ge\frac1{2n}\left[  \tr\left(  nI_{n}
+\frac12D_{ij}\right)  ^{2} +\tr\left(  nI_{n}-\frac
12D_{ij}\right)  ^{2} \right] \\
&  =\frac1{2n}\tr\left(  2n^{2}I_{n}+\frac12D_{ij}^{2}\right) \\
&  =n^{2}+\frac{\|D_{ij}\|_{\HS}^{2}}{4n},
\end{align*}
where the last equality uses the symmetry of $D_{ij}$.
\end{proof}

\begin{theorem}[Discrete spectral gap]
\label{thm:square-gap} Let
\[
A:\ell_{2}^{r}\times\ell_{2}^{n}\times\ell_{2}^{n}\longrightarrow\R
\]
be a trilinear form with coefficients in $\{-1,1\}$. Then either
\[
\|A\|=\sqrt n
\]
or
\[
\|A\|\ge\left(  n^{2}+\frac{2}{n}\right)  ^{1/4}.
\]
Consequently, either
\[
\Lambda(r,n,n)=\sqrt n
\]
or
\[
\Lambda(r,n,n)\ge\left(  n^{2}+\frac{2}{n}\right)  ^{1/4}.
\]

\end{theorem}

\begin{proof}
For $n=1$, the form has the shape
\[
A(x,y,z)=yz\sum_{i=1}^{r}\varepsilon_{i} x_{i}, \qquad\varepsilon_{i}
\in\{-1,1\},
\]
so $\|A\|=\sqrt r$. If this norm is larger than $1$, then $r\ge2$ and
\[
\|A\|^{4}=r^{2}\ge4>3=1^{2}+\frac2{1}.
\]
Assume henceforth that $n\ge2$ and $\|A\|>\sqrt n$. By the equality
characterization \eqref{eq:exact-matrix-relations}, at least one of those matrix relations fails.

Suppose first that $C_{i}\ne0$ for some $i$. The diagonal entries of $C_{i}$
are zero, because every row of $M_{i}$ has squared Euclidean norm $n$. Since
$C_{i}$ is a nonzero symmetric integer matrix, it has a nonzero off-diagonal
entry together with its symmetric counterpart. Hence
\[
\|C_{i}\|_{\HS}^{2}\ge2.
\]
\Cref{lem:square-gap-estimates}(a) gives
\[
\|A\|^{4}\ge\|M_{i}\|_{\op}^{4} \ge n^{2}+\frac2n.
\]

Consider now the case $C_{i}=0$ for every $i$. Since equality does
not hold, there are distinct indices $i,j$ with $D_{ij}\ne0$. Each $M_{i}$ is
then a Hadamard matrix. Because $n\ge2$, the orthogonality of two of its rows
implies that $n$ is even.

Every off-diagonal entry of $D_{ij}$ is a sum of two inner products of sign
vectors of length $n$. Each of these inner products is even, so every
off-diagonal entry of $D_{ij}$ is even. A diagonal entry has the form
\[
2\langle u,v\rangle,
\]
where $u,v\in\{-1,1\}^{n}$; it is therefore divisible by $4$. Since $D_{ij}$
is symmetric and nonzero, either it has a nonzero diagonal entry, contributing
at least $16$ to $\|D_{ij}\|_{\HS}^{2}$, or it has a nonzero
off-diagonal entry together with its symmetric counterpart, contributing at
least $8$. Thus
\[
\|D_{ij}\|_{\HS}^{2}\ge8.
\]
By \Cref{lem:square-gap-estimates}(b),
\[
\|A\|^{4}\ge L_{ij}^{4} \ge n^{2}+\frac{8}{4n} =n^{2}+\frac2n.
\]

\end{proof}

\begin{corollary}
 If
\[
\Lambda(r,n,n)<\left(  n^{2}+\frac2n\right)  ^{1/4},
\]
then
\[
\Lambda(r,n,n)=\sqrt n.
\]
Consequently, a Hadamard matrix of order $n$ exists and $r\le\rho(n)$. More
generally, if $\varepsilon\ge0$,
\[
\Lambda(r,n,n)\le(1+\varepsilon)\sqrt n
\]
and
\[
(1+\varepsilon)^{4}-1<\frac{2}{n^{3}},
\]
then the same conclusion holds.
\end{corollary}

\begin{proof}
The first assertion follows directly from \Cref{thm:square-gap}, and the
classification gives the final conditions. For the second assertion,
\[
(1+\varepsilon)^{4}n^{2}<n^{2}+\frac2n,
\]
so the upper bound lies strictly below the second alternative in the theorem.
\end{proof}

\begin{remark}
After the fourth powers are divided by $n^{2}$, the gap is $2/n^{3}$. In
particular, if $\varepsilon_{n}=o(n^{-3})$ and
\[
\Lambda(r_{n},n,n)\le(1+\varepsilon_{n})\sqrt n,
\]
then $\Lambda(r_{n},n,n)=\sqrt n$ for all sufficiently large $n$. The theorem does not cover the case in which the relative error tends to
zero at a slower rate.
\end{remark}

\begin{corollary}
For $n\in\N$,
\[
\Lambda(n,n,n)=\sqrt n \quad\Longleftrightarrow\quad n\in\{1,2,4,8\}.
\]
For every other $n$,
\[
\Lambda(n,n,n)\ge\left(  n^{2}+\frac{2}{n}\right)  ^{1/4}>\sqrt n.
\]

\end{corollary}

\begin{proof}
By \Cref{thm:exact-square}, equality is equivalent to the existence of a
Hadamard matrix of order $n$ together with $n\le\rho(n)$. We verify explicitly
when the latter inequality can occur. Write
\[
n=2^{4a+b}u, \qquad u\text{ odd}, \qquad0\le b\le3.
\]
If $a=0$, then
\[
\rho(n)=2^{b}\le2^{b}u=n,
\]
and equality holds exactly when $u=1$. This gives $n\in\{1,2,4,8\}$. If
$a\ge1$, then the inequality is strict. Indeed, for $a=1$ and $b=0$ one has
$\rho(n)=9<16\le n$; for $a=1$ and $b\ge1$,
\[
\rho(n)=8+2^{b}\le16<16\cdot2^{b}\le n;
\]
and for $a\ge2$,
\[
\rho(n)=8a+2^{b}\le8a+8<16^{a}\le2^{4a+b}\le n.
\]
Therefore $\rho(n)\le n$ for every $n$, with equality precisely for
$n\in\{1,2,4,8\}$. Hadamard matrices exist in these four orders. The strict
estimate follows from \Cref{thm:square-gap}.
\end{proof}

\section{Proof of Main Theorem A(iii): a lower bound for many slices}\label{sec:quantitative-obstruction}

We prove \hyperref[mainA]{Main Theorem~A}(iii) in four steps. First, the fourth-moment estimate
gives bounds for the matrix relations when the norm is close to $\sqrt n$.
A smaller family of slices is then selected, and the normalized slices are
replaced by nearby orthogonal matrices. Hermitian dilation produces involutions
that almost anticommute on average, while a finite-group Fourier argument gives
a lower bound on the dimension. Combining these estimates yields the stated
lower bound for $K_{r,n}$.

\subsection{Hilbert--Schmidt bounds for the matrix relations}

The fourth-moment estimate gives the following bounds when
$\|A\|\le(1+\varepsilon)\sqrt n$.

\begin{theorem}[Bounds from a near-minimal norm]\label{thm:analytic-defect}
Let $A$ be unimodular on $\ell_2^r\times\ell_2^m\times\ell_2^n$, where $m\le n$. If
\[
 \|A\|\le (1+\varepsilon)\sqrt n,
\]
then
\begin{equation}\label{eq:defect-upper}
 \Defect(M_1,\dots,M_r)
 \le mr(r+2)n^2\big((1+\varepsilon)^4-1\big).
\end{equation}
Consequently,
\begin{equation}\label{eq:slice-orth-stab}
 \|M_iM_i^T-nI_m\|_{\HS}
 \le n\sqrt{\frac{mr(r+2)}{2}}
 \sqrt{(1+\varepsilon)^4-1}
\end{equation}
and, for $i\ne j$,
\begin{equation}\label{eq:anti-stab}
 \|M_iM_j^T+M_jM_i^T\|_{\HS}
 \le n\sqrt{mr(r+2)}
 \sqrt{(1+\varepsilon)^4-1}.
\end{equation}
\end{theorem}

\begin{proof}
The hypothesis gives
\[
 \|A\|^4\le(1+\varepsilon)^4n^2.
\]
On the other hand, \Cref{thm:defect-lower} gives
\[
 \|A\|^4\ge n^2+\frac{\Defect(M_1,\ldots,M_r)}{mr(r+2)}.
\]
Subtracting $n^2$ and multiplying by $mr(r+2)$ yields
\[
 \Defect(M_1,\ldots,M_r)\le mr(r+2)n^2\bigl((1+\varepsilon)^4-1\bigr),
\]
which is \eqref{eq:defect-upper}. Since every term in
\eqref{eq:defect} is nonnegative,
\[
 2\|C_i\|_{\HS}^2\le\Defect(M_1,\ldots,M_r),
 \qquad
 \|D_{ij}\|_{\HS}^2\le\Defect(M_1,\ldots,M_r).
\]
Taking square roots and inserting \eqref{eq:defect-upper} gives
\eqref{eq:slice-orth-stab} and \eqref{eq:anti-stab}.
\end{proof}

In the square case we can get a better bound for each individual slice by
using its operator norm.

\begin{lemma}[A bound for each slice]\label{lem:spectral-slice}
Let $A$ have square slices and satisfy
$\|A\|\le(1+\varepsilon)\sqrt n$. Then, for every $i$,
\begin{equation}\label{eq:spectral-slice}
 \frac{\|M_iM_i^T-nI_n\|_{\HS}^2}{n^3}
 \le (1+\varepsilon)^2-1.
\end{equation}
\end{lemma}

\begin{proof}
Set $B_i=M_i/\sqrt n$. Since $M_i$ has $n^2$ entries of modulus one,
\[
 \|B_i\|_{\HS}^2=\frac1n\|M_i\|_{\HS}^2=n.
\]
Let $\lambda_1,\ldots,\lambda_n\ge0$ be the eigenvalues of
$B_iB_i^T$. Then
\begin{equation}
 \sum_{\nu=1}^n\lambda_\nu
 =\tr(B_iB_i^T)=\|B_i\|_{\HS}^2=n.
 \label{eq:slice-eigenvalue-sum}
\end{equation}
Moreover, choosing $x=e_i$ in the slice formula \eqref{eq:trilinear-norm-representation} gives
\[
 \|B_i\|_{\op}=\frac{\|M_i\|_{\op}}{\sqrt n}
 \le\frac{\|A\|}{\sqrt n}\le1+\varepsilon.
\]
Hence every eigenvalue satisfies
\[
 0\le\lambda_\nu\le\|B_iB_i^T\|_{\op}
 =\|B_i\|_{\op}^2\le(1+\varepsilon)^2.
\]
Therefore
$\lambda_\nu^2\le(1+\varepsilon)^2\lambda_\nu$. Summing and using
\eqref{eq:slice-eigenvalue-sum},
\[
 \sum_{\nu=1}^n\lambda_\nu^2
 \le(1+\varepsilon)^2n.
\]
Finally,
\begin{align*}
 \frac1n\|B_iB_i^T-I_n\|_{\HS}^2
 &=\frac1n\sum_{\nu=1}^n(\lambda_\nu-1)^2\\
 &=\frac1n\sum_{\nu=1}^n\lambda_\nu^2
   -\frac2n\sum_{\nu=1}^n\lambda_\nu+1\\
 &=\frac1n\sum_{\nu=1}^n\lambda_\nu^2-1\\
 &\le(1+\varepsilon)^2-1.
\end{align*}
Since
\[
 \frac1n\|B_iB_i^T-I_n\|_{\HS}^2
 =\frac{\|M_iM_i^T-nI_n\|_{\HS}^2}{n^3},
\]
this is \eqref{eq:spectral-slice}.
\end{proof}

\begin{remark}
For small $\varepsilon$,
\[
 (1+\varepsilon)^4-1=4\varepsilon+O(\varepsilon^2).
\]
Thus, for fixed $r$, the normalized errors in
\eqref{eq:slice-orth-stab} and \eqref{eq:anti-stab} are of order
$\sqrt\varepsilon$.
\end{remark}

\subsection{Selecting slices and using polar decomposition}

For square slices put
\[
 \Theta(\varepsilon)=(1+\varepsilon)^4-1
\]
and define
\[
 a_i=\frac{\|M_iM_i^T-nI_n\|_{\HS}^2}{n^3},
 \qquad
 b_{ij}=\frac{\|M_iM_j^T+M_jM_i^T\|_{\HS}^2}{n^3}.
\]
\Cref{thm:analytic-defect} gives
\begin{equation}\label{eq:total-pair-defect}
 \sum_{1\le i<j\le r}b_{ij}
 \le r(r+2)\Theta(\varepsilon).
\end{equation}
The diagonal terms admit the sharper pointwise bound
\begin{equation}\label{eq:diagonal-defect-bound}
 a_i\le(1+\varepsilon)^2-1
 \qquad(1\le i\le r)
\end{equation}
by \Cref{lem:spectral-slice}.

\begin{proposition}[Selecting a family of slices]\label{thm:average-extraction}
Let $A$ be a square-format sign tensor satisfying
$\|A\|\le(1+\varepsilon)\sqrt n$, and let $2\le q\le r$. Then there is a
set $S\subset\{1,\ldots,r\}$ with $|S|=q$ such that
\begin{equation}\label{eq:average-extraction}
 \frac1{\binom q2}
 \sum_{\substack{i<j\\i,j\in S}} b_{ij}
 \le \frac{2(r+2)}{r-1}\Theta(\varepsilon).
\end{equation}
Every selected slice also satisfies \eqref{eq:diagonal-defect-bound}.
\end{proposition}

\begin{proof}
Let $\mathcal S_q$ be the family of all $q$-element subsets of
$\{1,\ldots,r\}$, and choose $S$ uniformly from $\mathcal S_q$. For a fixed
pair $\{i,j\}$, the number of sets in $\mathcal S_q$ containing that pair is
$\binom{r-2}{q-2}$. Thus
\[
 \mathbb P(\{i,j\}\subset S)
 =\frac{\binom{r-2}{q-2}}{\binom rq}
 =\frac{q(q-1)}{r(r-1)}
 =\frac{\binom q2}{\binom r2}.
\]
By linearity of expectation,
\begin{align*}
 &\mathbb E_S\left[
 \frac1{\binom q2}\sum_{\substack{i<j\\i,j\in S}}b_{ij}
 \right]\\
 &\quad=\frac1{\binom q2}\sum_{i<j}b_{ij}
 \mathbb P(\{i,j\}\subset S)\\
 &\quad=\frac1{\binom r2}\sum_{i<j}b_{ij}\\
 &\quad\le\frac{r(r+2)}{\binom r2}\Theta(\varepsilon)
 =\frac{2(r+2)}{r-1}\Theta(\varepsilon),
\end{align*}
where \eqref{eq:total-pair-defect} was used in the inequality. Since a finite
average is at least its minimum, at least one $S\in\mathcal S_q$ satisfies
\eqref{eq:average-extraction}. The diagonal estimate
\eqref{eq:diagonal-defect-bound} holds for every slice before selection and hence for all
selected slices.
\end{proof}

Notice that \eqref{eq:average-extraction} has no extra factor of $q$. This is
used in the final estimate.

\subsubsection*{Polar approximation}

For a $d\times d$ matrix, write
\[
\|T\|_{2,d}=d^{-1/2}\|T\|_{\mathrm{HS}}.
\]
We denote by
\[
O(d)=\{Q\in M_d(\mathbb{R}) : Q^TQ=QQ^T=I_d\}
\]
the real orthogonal group of order $d$.

\begin{lemma}[Polar approximation]\label{lem:polar}
Let $B\in M_n(\R)$. Then there exists $Q\in O(n)$ such that
\begin{equation}
 \|B-Q\|_{\HS}\le \|BB^T-I_n\|_{\HS}.
\end{equation}
\end{lemma}

\begin{proof}
Let $s_1,\dots,s_n$ be the singular values of $B$, and let $Q$ be an orthogonal polar factor. Then
\[
 \|B-Q\|_{\HS}^2=\sum_{\nu=1}^n(s_\nu-1)^2.
\]
Since $s_\nu\ge0$,
\[
 |s_\nu-1|\le |s_\nu^2-1|.
\]
Therefore
\[
 \|B-Q\|_{\HS}^2
 \le\sum_{\nu=1}^n(s_\nu^2-1)^2
 =\|BB^T-I_n\|_{\HS}^2.
\]
\end{proof}

\begin{lemma}[Average anticommutator bound]\label{lem:average-polar}
Assume $0\le\varepsilon\le1$, and let $S$ be supplied by
\Cref{thm:average-extraction}. Put $B_i=M_i/\sqrt n$ for
$i\in S$. There are $Q_i\in O(n)$ such that
\[
 \|B_i-Q_i\|_{2,n}\le\sqrt{(1+\varepsilon)^2-1}
\]
and
\begin{equation}
 \frac1{\binom q2}\sum_{\substack{i<j\\i,j\in S}}
 \|Q_iQ_j^T+Q_jQ_i^T\|_{2,n}^2
 \le 80\Theta(\varepsilon),
\end{equation}
provided $r\ge6$.
\end{lemma}

\begin{proof}
For each selected index $i$, \Cref{lem:polar} gives an orthogonal
matrix $Q_i$ such that
\[
 \|B_i-Q_i\|_{2,n}
 \le\|B_iB_i^T-I_n\|_{2,n}.
\]
The square of the right-hand side is exactly $a_i$, so
\eqref{eq:diagonal-defect-bound} yields
\begin{equation}
 \delta_i:=\|B_i-Q_i\|_{2,n}
 \le\sqrt{(1+\varepsilon)^2-1}
 \le\sqrt{\Theta(\varepsilon)}.
 \label{eq:polar-slice-error}
\end{equation}

Fix distinct $i,j$. Insert and subtract $B_iQ_j^T$:
\begin{align*}
 Q_iQ_j^T-B_iB_j^T
 &=(Q_i-B_i)Q_j^T+B_i(Q_j-B_j)^T.
\end{align*}
The normalized Hilbert--Schmidt norm satisfies
$\|UV\|_{2,n}\le\|U\|_{2,n}\|V\|_{\op}$ and
$\|UV\|_{2,n}\le\|U\|_{\op}\|V\|_{2,n}$. Since $Q_j$ is orthogonal and
$\|B_i\|_{\op}\le1+\varepsilon\le2$,
\[
 \|Q_iQ_j^T-B_iB_j^T\|_{2,n}
 \le\delta_i+2\delta_j.
\]
Similarly,
\[
 \|Q_jQ_i^T-B_jB_i^T\|_{2,n}
 \le\delta_j+2\delta_i.
\]
Adding the two inequalities and using \eqref{eq:polar-slice-error},
\begin{equation}
 \begin{split}
 &\|Q_iQ_j^T+Q_jQ_i^T-(B_iB_j^T+B_jB_i^T)\|_{2,n}\\
 &\qquad\le3(\delta_i+\delta_j)
 \le6\sqrt{\Theta(\varepsilon)}.
 \end{split}
 \label{eq:pair-polar-error}
\end{equation}
Put
\[
 u_{ij}=\|B_iB_j^T+B_jB_i^T\|_{2,n},
 \qquad
 v_{ij}=\|Q_iQ_j^T+Q_jQ_i^T\|_{2,n}.
\]
By the triangle inequality and \eqref{eq:pair-polar-error},
$v_{ij}\le u_{ij}+6\sqrt{\Theta(\varepsilon)}$. Therefore
\[
 v_{ij}^2\le2u_{ij}^2+72\Theta(\varepsilon).
\]
Moreover,
\[
 u_{ij}^2
 =\frac1n\left\|\frac1n
 (M_iM_j^T+M_jM_i^T)\right\|_{\HS}^2=b_{ij}.
\]
Averaging over the selected pairs and applying
\Cref{thm:average-extraction},
\begin{align*}
 \frac1{\binom q2}\sum_{i<j}v_{ij}^2
 &\le2\frac1{\binom q2}\sum_{i<j}b_{ij}
   +72\Theta(\varepsilon)\\
 &\le\left(\frac{4(r+2)}{r-1}+72\right)
   \Theta(\varepsilon).
\end{align*}
If $r\ge6$, then $4(r+2)/(r-1)\le32/5<8$, so the last expression is at
most $80\Theta(\varepsilon)$.
\end{proof}

\subsection{Almost anticommuting matrices}

\begin{lemma}[Hermitian dilation]\label{lem:hermitian-dilation}
For $Q\in O(n)$ put
\[
 \mathcal X(Q)=
 \begin{pmatrix}0&Q\\Q^T&0\end{pmatrix}\in O(2n).
\]
Then $\mathcal X(Q)^*=\mathcal X(Q)$ and $\mathcal X(Q)^2=I_{2n}$. Moreover,
for $Q_i,Q_j\in O(n)$,
\begin{equation}\label{eq:dilation-anticommutator}
 \|\mathcal X(Q_i)\mathcal X(Q_j)+
 \mathcal X(Q_j)\mathcal X(Q_i)\|_{2,2n}^2
 =\|Q_iQ_j^T+Q_jQ_i^T\|_{2,n}^2,
\end{equation}
where $\|T\|_{2,d}=d^{-1/2}\|T\|_{\HS}$ for matrices of order $d$.
\end{lemma}

\begin{proof}
The first assertions are immediate. The anticommutator is block diagonal,
with diagonal blocks
\[
 A=Q_iQ_j^T+Q_jQ_i^T,
 \qquad
 B=Q_i^TQ_j+Q_j^TQ_i.
\]
Writing $U=Q_i^TQ_j$, one has $B=U+U^T$ and
$A=Q_i(U+U^T)Q_i^T$, so the two blocks have equal Hilbert--Schmidt norm.
Equation \eqref{eq:dilation-anticommutator} follows after normalization.
\end{proof}

For $q\ge1$ define the bilinear map
\[
 c(u,v)=\sum_{i>j}u_iv_j\in\mathbb F_2,
 \qquad u,v\in\mathbb F_2^q,
\]
and let $H_q=\mathbb F_2\times\mathbb F_2^q$ with multiplication
\begin{equation}
 (a,u)(b,v)=\bigl(a+b+c(u,v),u+v\bigr).
 \label{eq:clifford-extension-law}
\end{equation}

\begin{lemma}[A finite group with anticommuting generators]
The operation \eqref{eq:clifford-extension-law} is associative and makes
$H_q$ a group of order $2^{q+1}$. If $z=(1,0)$ and $x_i=(0,e_i)$, then
\[
 z^2=1,\qquad z\text{ is central},\qquad x_i^2=1,
 \qquad x_ix_j=zx_jx_i\quad(i\ne j).
\]
Every element has the unique normal form
$z^a x_1^{u_1}\cdots x_q^{u_q}$, with $a,u_i\in\{0,1\}$.
\end{lemma}

\begin{proof}
The only nontrivial point for associativity is the central coordinate. Let
\[
 (a,u),\qquad (b,v),\qquad (d,w)\in H_q.
\]
The central coordinate of
$((a,u)(b,v))(d,w)$ is
\[
 a+b+d+c(u,v)+c(u+v,w).
\]
The central coordinate of $(a,u)((b,v)(d,w))$ is
\[
 a+b+d+c(v,w)+c(u,v+w).
\]
Since $c$ is bilinear,
\[
 c(u,v)+c(u+v,w)
 =c(u,v)+c(u,w)+c(v,w)
 =c(v,w)+c(u,v+w),
\]
so the two products agree. The identity is $(0,0)$, and
\[
 (a,u)^{-1}=\bigl(a+c(u,u),u\bigr),
\]
which follows by direct substitution. Thus $H_q$ is a group, and its order
is the cardinality of the underlying set, namely $2^{q+1}$.

The element $z=(1,0)$ is central because $c(0,u)=c(u,0)=0$, and $z^2=1$.
Also $c(e_i,e_i)=0$, so $x_i^2=1$. If $i<j$, then
$c(e_i,e_j)=0$ while $c(e_j,e_i)=1$. Hence
\[
 x_ix_j=(0,e_i+e_j),
 \qquad
 x_jx_i=(1,e_i+e_j)=zx_ix_j,
\]
which is equivalent to $x_ix_j=zx_jx_i$. Finally, the map
$(a,u)\mapsto z^a x_1^{u_1}\cdots x_q^{u_q}$ is a bijection because both
sides have the same explicitly prescribed coordinates; this gives existence
and uniqueness of the normal form.
\end{proof}

\begin{lemma}[Representations with $\pi(z)=-I$]\label{lem:clifford-sector}
Let $\pi$ be an irreducible complex representation of $H_q$ such that
$\pi(z)=-I$. Then
\[
 \dim\pi= D_q:=2^{\lfloor q/2\rfloor}.
\]
\end{lemma}

\begin{proof}
Because $\pi(z)=-I$, the defining relations of $H_q$ become
\[
 \pi(x_i)^2=I,
 \qquad
 \pi(x_i)\pi(x_j)=-\pi(x_j)\pi(x_i)
 \quad(i\ne j).
\]
Thus the assignment $e_i\mapsto\pi(x_i)$ extends uniquely to a unital
algebra homomorphism from the complex Clifford algebra
$\mathrm{Cl}_q(\C)$, generated by $e_1,\ldots,e_q$ with
$e_i^2=1$ and $e_ie_j=-e_je_i$, into $\operatorname{End}(V_\pi)$.
Irreducibility of $\pi$ as an $H_q$-representation is equivalent to
irreducibility as a module over this algebra, because the image algebra is
generated by the same operators $\pi(x_i)$.

We use the following standard complex classification. With
\[
 \sigma_1=\begin{pmatrix}0&1\\1&0\end{pmatrix},\quad
 \sigma_2=\begin{pmatrix}0&-i\\i&0\end{pmatrix},\quad
 \sigma_3=\begin{pmatrix}1&0\\0&-1\end{pmatrix},
\]
one obtains the recursion
\[
 \mathrm{Cl}_{q+2}(\C)
 \cong \mathrm{Cl}_q(\C)\otimes M_2(\C)
\]
by sending the old generators to $e_j\otimes\sigma_3$ and the two new
generators to $1\otimes\sigma_1$ and $1\otimes\sigma_2$. Starting from
\[
 \mathrm{Cl}_0(\C)=\C,
 \qquad
 \mathrm{Cl}_1(\C)\cong\C\oplus\C,
\]
induction gives
\[
 \mathrm{Cl}_{2s}(\C)\cong M_{2^s}(\C),
 \qquad
 \mathrm{Cl}_{2s+1}(\C)
 \cong M_{2^s}(\C)\oplus M_{2^s}(\C).
\]
The irreducible modules of $M_k(\C)$ have dimension $k$, and each
simple summand in the odd case again has $k=2^s$. Therefore every
irreducible representation with $\pi(z)=-I$ has dimension
$2^{\lfloor q/2\rfloor}$. See, for example, \cite{Porteous} for the
structure theory of complex Clifford algebras.
\end{proof}

For $d\in\mathbb{N}$, we denote by $U(d)$ the group of $d\times d$
unitary complex matrices.

We use the following form of matrix Fourier analysis. All
expectations are uniform. The required facts are standard; see, for example,
\cite{Serre}. Let $G$ be a finite group
and let $\widehat G$ be a complete set of irreducible complex unitary
representations. For $f:G\to M_d(\C)$ and $\pi\in\widehat G$, define
the matrix Fourier block
\begin{equation}
 \Fblock_\pi(f)=\mathbb E_{g\in G}\bigl[f(g)\otimes\pi(g)\bigr]
 \in M_d(\C)\otimes M_{d_\pi}(\C),
 \qquad d_\pi=\dim\pi.
 \label{eq:Fourier-block-defined}
\end{equation}

\begin{lemma}[Matrix Plancherel and the cubic identity]

For every $f:G\to M_d(\C)$,
\begin{equation}
 \sum_{\pi\in\widehat G}d_\pi
 \|\Fblock_\pi(f)\|_{\HS}^2
 =\mathbb E_g\|f(g)\|_{\HS}^2.
 \label{eq:matrix-plancherel}
\end{equation}
Moreover,
\begin{equation}
 \mathbb E_{g,h}\tr
 \bigl(f(gh)^*f(g)f(h)\bigr)
 =\sum_{\pi\in\widehat G}d_\pi
 \tr\bigl(
 \Fblock_\pi(f)\Fblock_\pi(f)^*\Fblock_\pi(f)\bigr).
 \label{eq:cubic-fourier-identity}
\end{equation}
\end{lemma}

\begin{proof}
For matrix-valued functions $a,b:G\to M_d(\C)$ define the normalized
convolution
\[
 (a*b)(x)=\mathbb E_{y\in G}a(y)b(y^{-1}x).
\]
The choice of order is important because the values need not commute. A
change of variables $x=yz$ and the homomorphism property of $\pi$ give
\begin{align*}
 \Fblock_\pi(a*b)
 &=\mathbb E_{x,y}\bigl[a(y)b(y^{-1}x)\otimes\pi(x)\bigr]\\
 &=\mathbb E_{y,z}\bigl[a(y)b(z)\otimes\pi(y)\pi(z)\bigr]\\
 &=\Fblock_\pi(a)\Fblock_\pi(b).
\end{align*}

We also use the matrix-valued Plancherel inner-product identity
\begin{equation}
 \mathbb E_x\tr\bigl(a(x)^*b(x)\bigr)
 =\sum_{\pi\in\widehat G}d_\pi
  \tr\bigl(\Fblock_\pi(a)^*\Fblock_\pi(b)\bigr).
 \label{eq:matrix-Plancherel-inner}
\end{equation}
To verify it directly, expand the right-hand side:
\begin{align*}
 &\sum_{\pi}d_\pi\tr
 \left(
  \mathbb E_x[a(x)^*\otimes\pi(x)^*]
  \mathbb E_y[b(y)\otimes\pi(y)]
 \right)\\
 &\qquad=\mathbb E_{x,y}\tr(a(x)^*b(y))
 \sum_{\pi}d_\pi\tr(\pi(x^{-1}y)).
\end{align*}
The regular-character identity
\[
 \sum_{\pi\in\widehat G}d_\pi\tr(\pi(t))
 =|G|\,\mathbf 1_{\{t=e\}}
\]
reduces the last expression to
$\mathbb E_x\tr(a(x)^*b(x))$, proving
\eqref{eq:matrix-Plancherel-inner}. Taking $a=b=f$ gives
\eqref{eq:matrix-plancherel}.

For the cubic identity, observe first that
\begin{align*}
 \mathbb E_{g,h}\tr
 \bigl(f(gh)^*f(g)f(h)\bigr)
 &=\mathbb E_x\tr
 \bigl(f(x)^*(f*f)(x)\bigr).
\end{align*}
Indeed, for fixed $x$, the substitution $x=gh$ identifies the average in
$g$ with the convolution defining $(f*f)(x)$. Applying
\eqref{eq:matrix-Plancherel-inner} and the convolution formula above yields
\begin{align*}
 \mathbb E_x\tr
 \bigl(f(x)^*(f*f)(x)\bigr)
 &=\sum_{\pi}d_\pi\tr
 \bigl(\Fblock_\pi(f)^*\Fblock_\pi(f*f)\bigr)\\
 &=\sum_{\pi}d_\pi\tr
 \bigl(\Fblock_\pi(f)^*\Fblock_\pi(f)^2\bigr)\\
 &=\sum_{\pi}d_\pi\tr
 \bigl(\Fblock_\pi(f)\Fblock_\pi(f)^*\Fblock_\pi(f)\bigr),
\end{align*}
where the last equality is cyclicity of the trace. This proves
\eqref{eq:cubic-fourier-identity}.
\end{proof}

\begin{lemma}[Operator norm of a Fourier block]
\label{lem:fourier-block-op}
If $f:G\to U(d)$ and $\pi\in\widehat G$, then
\begin{equation}
 \|\Fblock_\pi(f)\|_{\op}\le\sqrt{\frac d{d_\pi}}.
 \label{eq:fourier-block-op}
\end{equation}
\end{lemma}

\begin{proof}
The block $\Fblock_\pi(f)$ need not be self-adjoint or normal, so we use two
independent unit vectors. Let $\xi,\eta\in\C^d\otimes V_\pi$ have
norm one. Choose Schmidt decompositions
\[
 \xi=\sum_{i=1}^{s}\alpha_i a_i\otimes u_i,
 \qquad
 \eta=\sum_{j=1}^{t}\beta_j b_j\otimes v_j,
\]
where $s,t\le\min\{d,d_\pi\}$, the families $(a_i)$ and $(b_j)$ are
orthonormal in $\C^d$, the families $(u_i)$ and $(v_j)$ are
orthonormal in $V_\pi$, and
\[
 \sum_{i=1}^{s}|\alpha_i|^2=\sum_{j=1}^{t}|\beta_j|^2=1.
\]
Using \eqref{eq:Fourier-block-defined},
\begin{align*}
 \langle\xi,\Fblock_\pi(f)\eta\rangle
 &=\mathbb E_g\sum_{i,j}
 \overline{\alpha_i}\beta_j
 \langle a_i,f(g)b_j\rangle
 \langle u_i,\pi(g)v_j\rangle.
\end{align*}
Apply Cauchy--Schwarz to the joint index $(g,i,j)$:
\begin{align}
 |\langle\xi,\Fblock_\pi(f)\eta\rangle|^2
 &\le
 \left(\mathbb E_g\sum_{i,j}
 |\langle a_i,f(g)b_j\rangle|^2\right)\\
 &\quad\times
 \left(\mathbb E_g\sum_{i,j}|\alpha_i|^2|\beta_j|^2
 |\langle u_i,\pi(g)v_j\rangle|^2\right).
 \label{eq:two-Schmidt-CS}
\end{align}
For the first factor, completing the two orthonormal systems to bases if
necessary,
\[
 \sum_{i,j}|\langle a_i,f(g)b_j\rangle|^2
 \le\|f(g)\|_{\HS}^2=d.
\]
For the second factor, Schur orthogonality for matrix coefficients of the
irreducible representation $\pi$ gives, for unit vectors $u,v$,
\[
 \mathbb E_g|\langle u,\pi(g)v\rangle|^2=\frac1{d_\pi}.
\]
Consequently the second factor in \eqref{eq:two-Schmidt-CS} equals
\[
 \frac1{d_\pi}\sum_{i,j}|\alpha_i|^2|\beta_j|^2
 =\frac1{d_\pi}.
\]
Thus
\[
 |\langle\xi,\Fblock_\pi(f)\eta\rangle|
 \le\sqrt{\frac d{d_\pi}}.
\]
Taking the supremum over independent unit vectors $\xi$ and $\eta$ proves
\eqref{eq:fourier-block-op}. The use of two vectors is essential:
restricting to $\langle\xi,\Fblock_\pi\xi\rangle$ would control only the
numerical radius for a general nonnormal block.
\end{proof}

\begin{lemma}[A Fourier estimate]
\label{lem:cubic-fourier-support}
Let $f:G\to U(d)$ and let $\Omega\subset\widehat G$ contain every
$\pi$ for which $\Fblock_\pi(f)\ne0$. If the trivial representation is not
in $\Omega$ and $D=\min_{\pi\in\Omega}d_\pi$, then
\begin{equation}
 \operatorname{Re}\mathbb E_{g,h}
 \tr\bigl(f(gh)^*f(g)f(h)\bigr)
 \le d\sqrt{\frac dD}.
\end{equation}
\end{lemma}

\begin{proof}
For any matrix $F$,
\[
 |\tr(FF^*F)|
 \le\|F\|_{\op}\|FF^*\|_1
 =\|F\|_{\op}\|F\|_{\HS}^2.
\]
Using the cubic identity, \Cref{lem:fourier-block-op}, and then
Plancherel,
\begin{align*}
 \operatorname{Re}\mathbb E_{g,h}
 \tr(f(gh)^*f(g)f(h))
 &\le\sum_{\pi\in\Omega}d_\pi
 \|\Fblock_\pi(f)\|_{\op}\|\Fblock_\pi(f)\|_{\HS}^2\\
 &\le\sqrt{\frac dD}
 \sum_{\pi\in\Omega}d_\pi\|\Fblock_\pi(f)\|_{\HS}^2\\
 &\le\sqrt{\frac dD}\,\mathbb E_g\|f(g)\|_{\HS}^2\\
 &=d\sqrt{\frac dD}.
\end{align*}
\end{proof}

\begin{proposition}[A Fourier lower bound]
\label{prop:central-character-obstruction}
Let $G$ be a finite group, let $z\in G$ be a central involution, and let
$f:G\to U(d)$ satisfy
\[
 f(zg)=-f(g)\qquad(g\in G).
\]
Then there exists an irreducible representation satisfying $\pi(z)=-I$.
If $D_-$ denotes the least dimension of such a representation, then
\begin{equation}\label{eq:central-character-lower}
 \mathbb E_{g,h}
 \|f(gh)-f(g)f(h)\|_{2,d}^2
 \ge2\left(1-\sqrt{\frac d{D_-}}\right).
\end{equation}
\end{proposition}

\begin{proof}
Pairing each $g$ with $zg$ gives
\[
 \mathbb E_gf(g)
 =\frac12\mathbb E_g\bigl(f(g)+f(zg)\bigr)=0,
\]
so the trivial Fourier block vanishes. More generally, for every irrep
$\pi$, changing variables $g\mapsto zg$ in
\eqref{eq:Fourier-block-defined} gives
\begin{align}
 \Fblock_\pi(f)
 &=\mathbb E_g[f(zg)\otimes\pi(zg)]\\
 &=-\mathbb E_g[f(g)\otimes\pi(z)\pi(g)]\\
 &=-(I_d\otimes\pi(z))\Fblock_\pi(f).
 \label{eq:central-fourier-support}
\end{align}
Since $z$ is central, Schur's lemma implies
$\pi(z)=\omega_\pi I$ for a scalar $\omega_\pi$. Because $z^2=1$ and $\pi$
is unitary, $\omega_\pi\in\{1,-1\}$. If $\pi(z)=I$, then
\eqref{eq:central-fourier-support} says
$\Fblock_\pi(f)=-\Fblock_\pi(f)$, hence the block is zero. By
Plancherel, at least one Fourier block is nonzero, because
$\mathbb E_g\|f(g)\|_{\HS}^2=d$. Thus there is at least one irreducible representation with $\pi(z)=-I$,
and every nonzero Fourier block comes from such a representation. These
representations have dimension at least $D_-$. \Cref{lem:cubic-fourier-support}
therefore gives
\[
 \operatorname{Re}\mathbb E_{g,h}
 \tr(f(gh)^*f(g)f(h))
 \le d\sqrt{\frac d{D_-}}.
\]
Finally, since both $f(gh)$ and $f(g)f(h)$ are unitary,
\begin{align*}
 \|f(gh)-f(g)f(h)\|_{\HS}^2
 &=\|f(gh)\|_{\HS}^2+\|f(g)f(h)\|_{\HS}^2\\
 &\quad-2\operatorname{Re}\tr
 \bigl(f(gh)^*f(g)f(h)\bigr)\\
 &=2d-2\operatorname{Re}\tr
 \bigl(f(gh)^*f(g)f(h)\bigr).
\end{align*}
Average and divide by $d$ to obtain \eqref{eq:central-character-lower}.
\end{proof}

\begin{lemma}[Error from reordering the words]
Let $q\ge2$, and let $X_1,\ldots,X_q\in U(d)$ be Hermitian involutions.
Put
\[
 e_{ij}=\|X_iX_j+X_jX_i\|_{2,d}\qquad(i<j).
\]
Define $f:H_q\to U(d)$ by
\[
 f(z^a x_1^{u_1}\cdots x_q^{u_q})
 =(-I_d)^aX_1^{u_1}\cdots X_q^{u_q}.
\]
If $N=\binom q2$, then
\begin{equation}\label{eq:normal-word-upper}
 \mathbb E_{g,h\in H_q}
 \|f(gh)-f(g)f(h)\|_{2,d}^2
 \le\frac N4\sum_{i<j}e_{ij}^2.
\end{equation}
Consequently,
\begin{equation}\label{eq:average-clifford-dimension}
 \frac1N\sum_{i<j}e_{ij}^2
 \ge\frac8{N^2}
 \left(1-\sqrt{\frac d{D_q}}\right).
\end{equation}
\end{lemma}

\begin{proof}
Write
\[
 g=z^a x_1^{u_1}\cdots x_q^{u_q},
 \qquad
 h=z^b x_1^{v_1}\cdots x_q^{v_q}.
\]
The central signs $(-I)^a$ and $(-I)^b$ are scalar unitaries and do not
affect any norm. Put
\[
 X^u=X_1^{u_1}\cdots X_q^{u_q},
 \qquad
 c(u,v)=\sum_{i>j}u_iv_j\pmod2.
\]
By the multiplication law in $H_q$,
\[
 f(gh)=(-1)^{a+b+c(u,v)}X^{u+v},
 \qquad
 f(g)f(h)=(-1)^{a+b}X^uX^v.
\]
Thus it suffices to compare $X^uX^v$ with
$(-1)^{c(u,v)}X^{u+v}$.

Start with the concatenated word
\[
 X_1^{u_1}\cdots X_q^{u_q}X_1^{v_1}\cdots X_q^{v_q}.
\]
To put it in increasing index order, a factor $X_j$ from the second word
must pass a factor $X_i$ from the first word precisely when
$i>j$, $u_i=1$, and $v_j=1$. Each ordered pair $(i,j)$ is used at most once.
At such a step the exact anticommutation replacement would be
$X_iX_j\mapsto-X_jX_i$, and the error of this replacement is
\[
 X_iX_j-(-X_jX_i)=X_iX_j+X_jX_i.
\]
All factors to the left and right of the replaced pair are unitary, so the
normalized Hilbert--Schmidt norm of the error is exactly $e_{ji}$.
After all interchanges, equal adjacent factors cancel exactly because
$X_i^2=I$. If $W_0,W_1,\ldots,W_L$ are the successive signed words, then
$W_0=X^uX^v$, $W_L=(-1)^{c(u,v)}X^{u+v}$, and telescoping gives
\begin{align}
 \|f(gh)-f(g)f(h)\|_{2,d}
 &=\|W_0-W_L\|_{2,d}\\
 &\le\sum_{t=1}^L\|W_{t-1}-W_t\|_{2,d}\\
 &\le\sum_{i>j}u_iv_j e_{ji}.
 \label{eq:normal-word-telescope}
\end{align}
There are at most $N=\binom q2$ possible pairs. Cauchy--Schwarz in the
finite sum gives
\begin{align*}
 \left(\sum_{i>j}u_iv_j e_{ji}\right)^2
 &\le\left(\sum_{i>j}u_iv_j\right)
      \left(\sum_{i>j}u_iv_j e_{ji}^2\right)\\
 &\le N\sum_{i>j}u_iv_j e_{ji}^2.
\end{align*}
Average over independent uniform $u,v\in\mathbb F_2^q$. For each fixed
$i>j$, the two bits $u_i$ and $v_j$ are independent Bernoulli variables of
mean $1/2$, so
\[
 \mathbb E_{u,v}(u_iv_j)=\frac14.
\]
Therefore \eqref{eq:normal-word-telescope} implies
\begin{align*}
 \mathbb E_{g,h}\|f(gh)-f(g)f(h)\|_{2,d}^2
 &\le N\sum_{i>j}\frac14e_{ji}^2
 =\frac N4\sum_{i<j}e_{ij}^2,
\end{align*}
which is \eqref{eq:normal-word-upper}.

The definition of $f$ gives $f(zg)=-f(g)$. \Cref{prop:central-character-obstruction}
and \Cref{lem:clifford-sector} therefore yield
\[
 \mathbb E_{g,h}\|f(gh)-f(g)f(h)\|_{2,d}^2
 \ge2\left(1-\sqrt{\frac d{D_q}}\right).
\]
Combining upper and lower estimates,
\[
 \frac N4\sum_{i<j}e_{ij}^2
 \ge2\left(1-\sqrt{\frac d{D_q}}\right).
\]
Multiplying by $4/N$ and then dividing by $N$ gives
\[
 \frac1N\sum_{i<j}e_{ij}^2
 \ge\frac8{N^2}
 \left(1-\sqrt{\frac d{D_q}}\right),
\]
which is \eqref{eq:average-clifford-dimension}.
\end{proof}

\begin{theorem}[Dimension bound for almost anticommuting matrices]
\label{thm:average-clifford-obstruction}
Let $q\ge2$, and let $X_1,\ldots,X_q\in U(d)$ be Hermitian
involutions. If $D_q=2^{\lfloor q/2\rfloor}\ge4d$, then
\begin{equation}
 \frac1{\binom q2}\sum_{i<j}
 \|X_iX_j+X_jX_i\|_{2,d}^2
 \ge\frac{4}{\binom q2^{\,2}}.
\end{equation}
\end{theorem}

\begin{proof}
Under the dimension hypothesis,
$1-\sqrt{d/D_q}\ge1/2$. Apply
\eqref{eq:average-clifford-dimension}.
\end{proof}

\subsection{A lower bound for the Hilbertian KSZ constant}

\begin{theorem}[Lower bound for large $r$]\label{thm:supercritical-gap}
There is a universal constant $c>0$ such that, for every $r,n\in\N$
with
\[
 r\ge q_n:=2\lceil\log_2(8n)\rceil,
\]
one has
\begin{equation}\label{eq:supercritical-gap}
 K_{r,n}-1\ge\frac{1}{75q_n^4}
 \ge\frac{c}{(1+\log_2n)^4}.
\end{equation}
Consequently, for every $\delta>0$ there is $n_\delta$ such that
\[
 r\ge(2+\delta)\log_2n,\qquad n\ge n_\delta
 \quad\Longrightarrow\quad
 K_{r,n}\ge1+\frac{c}{(1+\log_2n)^4}.
\]
\end{theorem}

\begin{proof}
Because the set of sign tensors of format $r\times n\times n$ is finite,
there is an extremizer $A$. Write
\[
 \|A\|=K_{r,n}\sqrt n=(1+\varepsilon)\sqrt n,
 \qquad \varepsilon=K_{r,n}-1\ge0.
\]
If $\varepsilon>1$, then the first lower bound in
\eqref{eq:supercritical-gap} is immediate, because $q_n\ge1$. Hence we may
assume $0\le\varepsilon\le1$.

Set
\[
 q=q_n=2\lceil\log_2(8n)\rceil,
 \qquad N=\binom q2.
\]
The hypothesis $r\ge q$ allows us to apply
\Cref{thm:average-extraction}. Since $n\ge1$ implies
$q\ge2\lceil3\rceil=6$, we also have $r\ge6$, so
\Cref{lem:average-polar} applies. We obtain orthogonal matrices
$Q_1,\ldots,Q_q\in O(n)$ satisfying
\begin{equation}
 \frac1N\sum_{i<j}
 \|Q_iQ_j^T+Q_jQ_i^T\|_{2,n}^2
 \le80\Theta(\varepsilon).
\end{equation}

For each $i$, put
\[
 X_i=\mathcal X(Q_i)=
 \begin{pmatrix}0&Q_i\\Q_i^T&0\end{pmatrix}\in O(2n).
\]
By \Cref{lem:hermitian-dilation}, the $X_i$ are Hermitian involutions
and
\begin{equation}
 \frac1N\sum_{i<j}\|X_iX_j+X_jX_i\|_{2,2n}^2
 \le80\Theta(\varepsilon).
 \label{eq:clifford-average-upper}
\end{equation}
Since $q$ is even,
\[
 D_q=2^{\lfloor q/2\rfloor}=2^{q/2}
 =2^{\lceil\log_2(8n)\rceil}\ge8n.
\]
The matrices $X_i$ have order $d=2n$, and therefore
\[
 D_q\ge8n=4d.
\]
\Cref{thm:average-clifford-obstruction} gives the lower bound
\begin{equation}
 \frac1N\sum_{i<j}\|X_iX_j+X_jX_i\|_{2,2n}^2
 \ge\frac4{N^2}.
 \label{eq:clifford-average-lower}
\end{equation}
Comparing \eqref{eq:clifford-average-upper} and
\eqref{eq:clifford-average-lower},
\[
 \frac4{N^2}\le80\Theta(\varepsilon),
 \qquad\text{hence}\qquad
 \Theta(\varepsilon)\ge\frac1{20N^2}.
\]
For $0\le\varepsilon\le1$,
\begin{align*}
 \Theta(\varepsilon)
 &=(1+\varepsilon)^4-1\\
 &=4\varepsilon+6\varepsilon^2+4\varepsilon^3+\varepsilon^4\\
 &\le(4+6+4+1)\varepsilon=15\varepsilon.
\end{align*}
Consequently,
\[
 \varepsilon\ge\frac1{300N^2}.
\]
Since $N=q(q-1)/2\le q^2/2$,
\[
 \frac1{300N^2}
 \ge\frac1{300(q^4/4)}=\frac1{75q^4}.
\]
This establishes the first inequality in \eqref{eq:supercritical-gap}. For the second,
\[
 q_n=2\lceil\log_2(8n)\rceil
 \le2(\log_2(8n)+1)=2\log_2n+8
 \le8(1+\log_2n),
\]
where the last inequality is valid for $n\ge1$. Hence
\[
 \frac{1}{75q_n^4}
 \ge\frac{1}{75\cdot8^4(1+\log_2n)^4}.
\]
Thus the second inequality holds with
$c=(75\cdot8^4)^{-1}$.

For the final assertion, it is enough that
\[
 \delta\log_2n\ge8,
\]
because then
\[
 (2+\delta)\log_2n\ge2\log_2n+8\ge q_n.
\]
Thus one may take, for instance, $n_\delta=\lceil 2^{8/\delta}\rceil$.
\end{proof}

\begin{remark}[On the exponent four]
We do not claim that the exponent $4$ is optimal. In the proof,
$N=\binom q2$ is of order $q^2$, and the estimate is of order $N^{-2}$.
Hence the argument gives $q^{-4}$ when $q$ is of order $\log_2 n$.
\end{remark}

\begin{remark}
The two estimates are different. The gap in \Cref{thm:square-gap} comes from
integrality and separates the exact minimum from every other value. The bound
in \Cref{thm:supercritical-gap} is used when the number of slices is large.
\end{remark}

\section{Proof of Main Theorem A(ii): the asymptotic construction}\label{sec:subcritical-construction}

Exact forms are first constructed in Hadamard dimensions. Choosing nearby
Hadamard orders and restricting the last variable then gives \hyperref[mainA]{Main Theorem~A}(ii).

\subsection{Orthogonal multiplications and exact constructions}\label{sec:orthogonal-multiplications}

We use the equality condition in the following form.

\begin{proposition}[Orthogonal-multiplication characterization]\label{prop:orthogonal-multiplication}
Let $r,m,n$ be positive integers with $m\le n$. Then $\Lambda(r,m,n)=\sqrt n$
if and only if there exist sign matrices $M_{1},\ldots,M_{r}\in
\{-1,1\}^{m\times n}$ satisfying
\[
M_{i}M_{j}^{T}+M_{j}M_{i}^{T}=2n\delta_{ij}I_{m}.
\]
Equivalently, the bilinear map
\[
\omega:\R^{r}\times\R^{m}\longrightarrow\R^{n},
\qquad\omega(x,y)=M(x)^{T}y,
\]
satisfies
\[
\|\omega(x,y)\|_{2}=\sqrt n\,\|x\|_{2}\|y\|_{2}
\]
and has sign coefficients in the canonical bases.
\end{proposition}

\begin{proof}
The matrix characterization follows from \Cref{sec:fourth-moment}. If the matrix relations
hold, then
\[
M(x)M(x)^{T}=n\|x\|_{2}^{2}I_{m},
\]
and consequently
\begin{align*}
\|\omega(x,y)\|_{2}^{2}  &  =\|M(x)^{T}y\|_{2}^{2}\\
&  =\langle M(x)M(x)^{T}y,y\rangle\\
&  =n\|x\|_{2}^{2}\|y\|_{2}^{2}.
\end{align*}
Conversely, if the norm identity holds for all $x$ and $y$, then
\[
y^{T}M(x)M(x)^{T}y=n\|x\|_{2}^{2}\|y\|_{2}^{2}
\]
for every $y$. Hence $M(x)M(x)^{T}=n\|x\|_{2}^{2}I_{m}$ for every $x$.
Comparing the coefficients of $x_{i}^{2}$ and $x_{i}x_{j}$ gives the stated
matrix relations.
\end{proof}

The classical norm-composition problem considers bilinear maps that preserve
a product of Euclidean norms; see \cite{Hurwitz,Shapiro}. Here we also require
sign coefficients in the canonical bases.

\begin{theorem}[Exact rectangular construction]
Let $r,m,n$ be positive integers with $m\le n$.
Suppose that a Hadamard matrix of order $n$ exists and that $r\le\rho(n)$.
Then
\[
\Lambda(r,m,n)=\sqrt n.
\]

\end{theorem}

\begin{proof}
The proof of the square classification provides sign matrices $S_{1}
,\ldots,S_{r}\in\{-1,1\}^{n\times n}$ satisfying
\[
S_{i}S_{j}^{T}+S_{j}S_{i}^{T}=2n\delta_{ij}I_{n}.
\]
Let $P_{m}:\R^{n}\to\R^{m}$ be the projection onto the first
$m$ coordinates and set
\[
M_{i}=P_{m}S_{i}.
\]
Then $M_{i}\in\{-1,1\}^{m\times n}$ and
\begin{align*}
M_{i}M_{j}^{T}+M_{j}M_{i}^{T}  &  =P_{m}(S_{i}S_{j}^{T}+S_{j}S_{i}^{T}
)P_{m}^{T}\\
&  =2n\delta_{ij}I_{m}.
\end{align*}
\Cref{prop:orthogonal-multiplication} completes the proof.
\end{proof}

\begin{remark}
When $m<n$, the converse existence problem asks which triples $(r,m,n)$ admit
sign matrices satisfying \eqref{eq:exact-matrix-relations}. In the square case, one matrix attaining
the lower bound is necessarily Hadamard and the remaining matrices reduce to a
Hurwitz--Radon family. This reduction is not available in general in the
rectangular case.
\end{remark}

\subsection{Hadamard orders near a target dimension}\label{sec:hadamard-approximation}

For the asymptotic result we need Hadamard orders close to a prescribed
number and with enough powers of two. The orders $4^u 12^v$ give the required
approximation, and an extra power of two controls the Hurwitz--Radon
function.

\begin{lemma}[An elementary approximation lemma]\label{lem:semigroup-gaps}
Let $\alpha,\beta>0$ and suppose $\alpha/\beta\notin\mathbb{Q}$. Put
\[
\mathcal{S}=\{u\alpha+v\beta:u,v\in\N\cup\{0\}\}.
\]
Then
\[
\inf\{s-x:s\in\mathcal{S},\ s\ge x\}\longrightarrow0 \qquad(x\to\infty).
\]

\end{lemma}

\begin{proof}
Fix $\varepsilon>0$, with $\varepsilon<\beta$. The residues $\{u\alpha
\pmod\beta:u\ge0\}$ are dense in $\R/\beta\mathbb{Z}$. For each $u\ge0$,
let $I_u$ be the set of residues $t$ from which the positively oriented
distance to $u\alpha$ modulo $\beta$ is smaller than $\varepsilon$. Density
implies that the open arcs $I_u$ cover $\R/\beta\mathbb{Z}$; compactness gives
a finite subcover $I_{u_1},\ldots,I_{u_N}$. Thus, for every residue $t$, some
$u_j\alpha$ lies at positively oriented distance in $[0,\varepsilon)$ from
$t$ modulo $\beta$.

Let $X=\max_{j} u_{j}\alpha$ and take $x\ge X$. Choose $j$ and $\delta
\in[0,\varepsilon)$ such that
\[
u_{j}\alpha-x\equiv\delta\pmod\beta.
\]
Thus $u_{j}\alpha-x=\delta-q\beta$ for some integer $q$. Since $u_{j}\alpha\le
x$ and $\delta<\beta$, one has $q\ge0$. Therefore
\[
u_{j}\alpha+q\beta=x+\delta\in[x,x+\varepsilon),
\]
and the left-hand side belongs to $\mathcal{S}$. Hence every sufficiently
large interval $[x,x+\varepsilon)$ meets $\mathcal{S}$.
\end{proof}

\begin{proposition}[Hadamard orders close to a given size]\label{prop:hadamard-relative-approximation}
Let $(x_{n})$ be any sequence of positive real numbers with $x_{n}\to\infty$.
Then there are Hadamard orders $h_{n}\ge x_{n}$ such that
\[
\frac{h_{n}}{x_{n}}\longrightarrow1.
\]

\end{proposition}

\begin{proof}
Every element of
\[
\mathcal{H}=\{4^{u}12^{v}:u,v\in\N\cup\{0\}\}
\]
is a Hadamard order. The logarithms of the elements of $\mathcal{H}$ form the
additive semigroup in \Cref{lem:semigroup-gaps} with $\alpha=\log4$ and $\beta
=\log12$. The quotient $\alpha/\beta$ is irrational: otherwise $\log_{2} 3$
would be rational, which would imply $2^{a}=3^{b}$ for some positive integers
$a,b$.

The set $\mathcal{H}$ has only finitely many elements below each fixed bound,
so the least $h_{n}\in\mathcal{H}$ with $h_{n}\ge x_{n}$ is well defined. The
lemma gives
\[
\log h_{n}-\log x_{n}\longrightarrow0,
\]
and hence $h_{n}/x_{n}\to1$.
\end{proof}

\begin{proposition}[Hadamard orders with a Hurwitz--Radon bound]\label{prop:hadamard-rho-approximation}
Let $(a_{n})$ be a sequence of nonnegative integers such that
\[
\frac{n}{2^{a_{n}}}\longrightarrow\infty.
\]
Then there are Hadamard orders $t_{n}\ge n$ satisfying
\[
\frac{t_{n}}{n}\longrightarrow1 \qquad\text{and}\qquad\rho(t_{n})\ge2a_{n}.
\]

\end{proposition}

\begin{proof}
Put $x_{n}=n/2^{a_{n}}$ and choose $h_{n}$ as in
\Cref{prop:hadamard-relative-approximation}.
Since powers of two are Hadamard orders, so is
\[
t_{n}=2^{a_{n}}h_{n}.
\]
Moreover,
\[
t_{n}\ge n, \qquad\frac{t_{n}}{n}=\frac{h_{n}}{x_{n}}\longrightarrow1.
\]
Finally, by \eqref{eq:rho-valuation-bound},
\[
\rho(t_{n})\ge2\nu_{2}(t_{n})\ge2a_{n}.
\]

\end{proof}

\begin{corollary}
For every $0<c<2$, there is a sequence of Hadamard orders $(t_{n})$ such that
\[
n\le t_{n}, \qquad\frac{t_{n}}{n}\longrightarrow1, \qquad\rho(t_{n})\ge
c\log_{2}n
\]
for all sufficiently large $n$.
\end{corollary}

\begin{proof}
Take
\[
a_{n}=\left\lceil \frac c2\log_{2}n\right\rceil .
\]
Since $a_{n}\le\frac c2\log_{2}n+1$,
\[
\frac{n}{2^{a_{n}}} \ge\frac12n^{1-c/2}\longrightarrow\infty.
\]
\Cref{prop:hadamard-rho-approximation} gives
$\rho(t_{n})\ge2a_{n}\ge c\log_{2}n$.
\end{proof}

\subsection{The asymptotic construction}

The exact construction can now be combined with the approximation of Hadamard
orders.

\begin{theorem}[Asymptotic construction below $2\log_2 n$]\label{thm:subcritical-asymptotic}
Let $(r_{n})$ and $(m_{n})$ be sequences of positive integers such that
\[
1\le m_{n}\le n
\]
and
\begin{equation}
\limsup_{n\to\infty}\frac{r_{n}}{\log_{2}n}<2. \label{eq:trilinear-subcritical-condition}
\end{equation}
Then
\[
\Lambda(r_{n},m_{n},n)=(1+o(1))\sqrt n.
\]
More precisely, there exist unimodular trilinear forms
\[
A_{n}:\ell_{2}^{r_{n}}\times\ell_{2}^{m_{n}}\times\ell_{2}^{n}\longrightarrow
\R
\]
such that
\[
\sqrt n\le\|A_{n}\|\le(1+o(1))\sqrt n.
\]

\end{theorem}

\begin{proof}
Choose $c$ such that
\[
\limsup_{n\to\infty}\frac{r_{n}}{\log_{2}n}<c<2.
\]
Then $r_{n}\le c\log_{2}n$ for all sufficiently large $n$. Set
\[
a_{n}=\left\lceil \frac{r_{n}}{2}\right\rceil .
\]
Since $a_{n}\le r_{n}/2+1$, eventually
\[
\frac{n}{2^{a_{n}}} \ge\frac12n^{1-c/2}\longrightarrow\infty.
\]
\Cref{prop:hadamard-rho-approximation} therefore gives Hadamard orders
$t_{n}\ge n$ such
that
\[
\frac{t_{n}}{n}\longrightarrow1 \qquad\text{and}\qquad\rho(t_{n})\ge2a_{n}\ge
r_{n}.
\]
The rectangular construction provides sign matrices
\[
S_{1}^{(n)},\ldots,S_{r_{n}}^{(n)}\in\{-1,1\}^{m_{n}\times t_{n}}
\]
such that, for every $x\in\R^{r_{n}}$,
\[
\left\|  \sum_{i=1}^{r_{n}}x_{i}S_{i}^{(n)}\right\|  _{\op}
=\sqrt{t_{n}}\,\|x\|_{2}.
\]
Let $J_{n}:\R^{n}\to\R^{t_{n}}$ be the canonical coordinate
embedding and define
\[
M_{i}^{(n)}=S_{i}^{(n)}J_{n}.
\]
These are $m_{n}\times n$ sign matrices. For every $x$,
\begin{align*}
\left\|  \sum_{i=1}^{r_{n}}x_{i}M_{i}^{(n)}\right\|  _{\op}  &
=\left\|  \left(  \sum_{i=1}^{r_{n}}x_{i}S_{i}^{(n)}\right)  J_{n}\right\|
_{\op}\\
&  \le\left\|  \sum_{i=1}^{r_{n}}x_{i}S_{i}^{(n)}\right\|  _{\op}
\|J_{n}\|_{\op}\\
&  =\sqrt{t_{n}}\,\|x\|_{2},
\end{align*}
because $J_{n}$ is an isometry. The corresponding form $A_{n}$ therefore
satisfies
\[
\|A_{n}\|\le\sqrt{t_{n}}.
\]
The universal lower bound gives $\|A_{n}\|\ge\sqrt n$. Hence
\[
1\le\frac{\|A_{n}\|}{\sqrt n} \le\sqrt{\frac{t_{n}}{n}}\longrightarrow1.
\]

\end{proof}

\begin{corollary}
Under the assumptions of the theorem, for every $\varepsilon>0$ there is
$N_{\varepsilon}$ such that for every $n\ge N_{\varepsilon}$ there exists a
trilinear form with sign coefficients and
\[
\|A_{n}\|\le(1+\varepsilon) \max\{\sqrt{r_{n}},\sqrt{m_{n}},\sqrt n\}
=(1+\varepsilon)\sqrt n.
\]
The constant $1$ is optimal.
\end{corollary}

\begin{proof}
Condition \eqref{eq:trilinear-subcritical-condition} implies $r_{n}\le n$ for all sufficiently large $n$,
while $m_{n}\le n$ by assumption. The estimate follows from the theorem, and
the lower bound $\sqrt n$ proves optimality.
\end{proof}

\begin{remark}
The number $2$ in \eqref{eq:trilinear-subcritical-condition} comes from the size of the exact
Hurwitz--Radon families. For powers of two,
\[
\rho(2^{a})=2a+O(1)=2\log_{2}(2^{a})+O(1).
\]
Thus exact Hurwitz--Radon families allow $r_n$ to have order
$2\log_{2}n$. The present methods do not determine whether
$K_{r_n,n}=1+o(1)$ can hold for larger values of $r_n$.
\end{remark}

\begin{corollary}
For every fixed $0<c<2$ and every sequence $1\le m_{n}\le n$,
\[
\Lambda\bigl(\max\{1,\left\lfloor c\log_{2}n\right\rfloor \},m_{n},n\bigr)
=(1+o(1))\sqrt n.
\]

\end{corollary}
\begin{remark}
The number $2$ describes the asymptotically maximal Hurwitz--Radon scale for
exact equality in the square case (attained, for instance, along powers of
two), but condition \eqref{eq:trilinear-subcritical-condition} is not
necessary for every rectangular format. If $H$ is a
Hadamard matrix of order $n$ and its rows are regarded as $1\times n$ matrices
$M_1,\ldots,M_n$, then
\[
M_iM_j^T+M_jM_i^T=2n\delta_{ij}.
\]
Hence $\Lambda(n,1,n)=\sqrt n$. The logarithmic restriction in the theorem is
therefore a feature of the Hurwitz--Radon construction used for general
$m_n$, not a general obstruction for rectangular equality.
\end{remark}

\medskip
\begin{proof}[Proof of \Cref{mainA}]
For part~\emph{(i)}, \Cref{thm:exact-square} gives
\[
 K_{r,n}=1
 \quad\Longleftrightarrow\quad
 \text{a Hadamard matrix of order $n$ exists and }r\le\rho(n),
\]
and \Cref{thm:square-gap} gives, for every square sign form with $\|A\|>\sqrt n$,
\[
 \frac{\|A\|}{\sqrt n}
 \ge \left(1+\frac{2}{n^3}\right)^{1/4}.
\]

Part~\emph{(ii)} is exactly \Cref{thm:subcritical-asymptotic}, after dividing
its estimate by $\sqrt n$.

Part~\emph{(iii)} is \Cref{thm:supercritical-gap}. The constant appearing there
is universal, and its second inequality gives the stated form with
$(1+\log_2 n)^{-4}$. The three assertions follow.
\end{proof}

\section{Proof of Main Theorem B}\label{sec:higher-order-defect}

We prove \hyperref[thm:intro-higher-order-main]{Main Theorem~B} by keeping the first variables separate, which
preserves the orthogonal-multiplication structure of the trilinear case. We
first record a quadrilinear extremizer and then integrate the fourth-moment
estimate over a product of spheres.

For $k\ge3$ and positive integers $d_{1},\ldots,d_{k}$, a \emph{sign
$k$-tensor} is an element
\[
\mathcal{E}\in\{-1,1\}^{d_{1}\times\cdots\times d_{k}}.
\]
It is identified with its associated real unimodular $k$-linear form on
$\ell_{2}^{d_{1}}\times\cdots\times\ell_{2}^{d_{k}}$, and its spectral norm is
the norm of that form. Let $\Lambda_{k}(d_{1},\ldots,d_{k})$ denote the least
spectral norm of a sign $k$-tensor of this format. Fixing every variable
except the $j$th one at coordinate vectors gives
\begin{equation}
\Lambda_{k}(d_{1},\ldots,d_{k})\ge\max_{1\le j\le k}\sqrt{d_{j}}. \label{eq:universal-tensor-lower-bound}
\end{equation}

\begin{example}[A concrete quadrilinear extremizer]\label{ex:quadrilinear-coordinate}
Let $x=(x_1,x_2)$, $y=(y_1,y_2)$ and $u,v\in\R^4$. Put
\[
 a=x_1+x_2,\qquad b=x_1-x_2,\qquad
 c=y_1+y_2,\qquad d=y_1-y_2,
\]
and define
\begin{align*}
 Q(x,y,u,v)={}&u_1(acv_1+adv_2+bcv_3+bdv_4)\\
 &+u_2(adv_1-acv_2+bdv_3-bcv_4)\\
 &+u_3(bcv_1+bdv_2-acv_3-adv_4)\\
 &+u_4(bdv_1-bcv_2-adv_3+acv_4).
\end{align*}
After expanding $a,b,c,d$, every coefficient of
$x_i y_j u_s v_t$ is either $1$ or $-1$. Thus $Q$ is a sign quadrilinear form
on
\[
 \ell_2^2\times\ell_2^2\times\ell_2^4\times\ell_2^4.
\]
If $W_1,\ldots,W_4$ denote the four expressions multiplying
$u_1,\ldots,u_4$, a direct calculation gives
\[
 W_1^2+W_2^2+W_3^2+W_4^2
 =(a^2+b^2)(c^2+d^2)\|v\|_2^2
 =4\|x\|_2^2\|y\|_2^2\|v\|_2^2.
\]
Consequently, Cauchy--Schwarz yields
\[
 |Q(x,y,u,v)|\le
 2\|x\|_2\|y\|_2\|u\|_2\|v\|_2.
\]
Equality is attained at
$x=y=v=e_1$ and $u=(1,1,1,1)/2$. Hence
\[
 \|Q\|=2=\sqrt4.
\]
This is the coordinate form of an exact quadrilinear orthogonal
multiplication. In particular, the quantity $\delta_k$ in
\Cref{thm:higher-order-defect} is zero for $Q$.
\end{example}

\subsection{The higher-order fourth-moment estimate}
Fix $k\ge3$ and put $q=k-2$. Consider a sign form
\[
 A:\ell_2^{r_1}\times\cdots\times\ell_2^{r_q}
 \times\ell_2^m\times\ell_2^n\longrightarrow\R,
 \qquad m\le n.
\]
For a multi-index
$\alpha=(\alpha_1,\ldots,\alpha_q)$ with
$1\le\alpha_\nu\le r_\nu$, let
\[
 T_\alpha\in\mathcal L(\ell_2^n,\ell_2^m)
\]
be the coefficient operator obtained by fixing the first $q$ variables at
the corresponding coordinate vectors. Thus every canonical matrix entry of
$T_\alpha$ belongs to $\{-1,1\}$. The associated operator-valued multilinear
map is
\begin{equation}\label{eq:higher-coefficient-map}
 \Phi_A(x^{(1)},\ldots,x^{(q)})
 =\sum_{\alpha_1=1}^{r_1}\cdots\sum_{\alpha_q=1}^{r_q}
 x^{(1)}_{\alpha_1}\cdots x^{(q)}_{\alpha_q}T_\alpha.
\end{equation}
For $y\in\ell_2^m$ and $z\in\ell_2^n$,
\[
 A(x^{(1)},\ldots,x^{(q)},y,z)
 =\langle \Phi_A(x^{(1)},\ldots,x^{(q)})z,y\rangle,
\]
and therefore
\begin{equation}\label{eq:higher-norm-representation}
 \|A\|
 =\sup_{\|x^{(1)}\|_2=\cdots=\|x^{(q)}\|_2=1}
 \|\Phi_A(x^{(1)},\ldots,x^{(q)})\|_{\op}.
\end{equation}
Let
\[
 \Omega=S^{r_1-1}\times\cdots\times S^{r_q-1}
\]
and let $\mu$ be the product of the normalized rotation-invariant probability
measures on these spheres. When $r_\nu=1$, the corresponding factor is
$S^0=\{-1,1\}$ with its uniform probability measure. Define
\begin{equation}\label{eq:higher-normalized-defect-definition}
 \delta_k(A):=\frac1{mn^2}\int_\Omega
 \left\|\Phi_A(x)\Phi_A(x)^T-nI_m\right\|_{\HS}^2\,d\mu(x).
\end{equation}

\begin{theorem}[Higher-order fourth-moment estimate]
\label{thm:higher-order-defect}
Every sign form as above satisfies
\begin{equation}\label{eq:higher-order-normalized-norm}
 \frac{\|A\|}{\sqrt n}\ge\bigl(1+\delta_k(A)\bigr)^{1/4}.
\end{equation}
Equivalently,
\begin{equation}\label{eq:higher-order-defect}
 \|A\|^4\ge n^2+\frac1m\int_\Omega
 \left\|
 \Phi_A(x)\Phi_A(x)^T-nI_m
 \right\|_{\HS}^2\,d\mu(x).
\end{equation}
Moreover, equality $\|A\|=\sqrt n$ holds if and only if
\begin{equation}\label{eq:higher-orthogonal-multiplication}
 \Phi_A(x)\Phi_A(x)^T
 =n\prod_{\nu=1}^q\|x^{(\nu)}\|_2^2 I_m
\end{equation}
for all $x^{(\nu)}\in\ell_2^{r_\nu}$.
\end{theorem}

\begin{proof}
Fix $x=(x^{(1)},\ldots,x^{(q)})\in\Omega$ and write
\[
 S(x)=\Phi_A(x)\Phi_A(x)^T\in\mathcal L(\ell_2^m).
\]
The operator $S(x)$ is positive semidefinite. If
$\lambda_1(x),\ldots,\lambda_m(x)$ are its eigenvalues, then
\begin{align}
 \|\Phi_A(x)\|_{\op}^4
 &=\lambda_{\max}(S(x))^2\notag\\
 &\ge \frac1m\sum_{j=1}^m\lambda_j(x)^2
 =\frac1m\tr(S(x)^2).\label{eq:higher-eigenvalue-step}
\end{align}
By \eqref{eq:higher-norm-representation}, the supremum of the left-hand side
of \eqref{eq:higher-eigenvalue-step} over $\Omega$ is $\|A\|^4$. Since $\mu$
is a probability measure,
\begin{equation}\label{eq:higher-average-step}
 \|A\|^4\ge\frac1m\int_\Omega\tr(S(x)^2)\,d\mu(x).
\end{equation}

We compute the mean trace of $S(x)$ next. For each sphere factor,
rotation invariance gives
\[
 \int_{S^{r_\nu-1}}x_i^{(\nu)}x_j^{(\nu)}\,d\sigma_\nu(x^{(\nu)})
 =\frac{\delta_{ij}}{r_\nu}.
\]
Expanding \eqref{eq:higher-coefficient-map} and using Fubini's theorem,
\begin{align*}
 \int_\Omega\tr S(x)\,d\mu(x)
 &=\sum_{\alpha,\beta}
 \left(\prod_{\nu=1}^q
 \int_{S^{r_\nu-1}}
 x^{(\nu)}_{\alpha_\nu}x^{(\nu)}_{\beta_\nu}\,d\sigma_\nu\right)
 \tr(T_\alpha T_\beta^T)\\
 &=\frac1{r_1\cdots r_q}
 \sum_\alpha\|T_\alpha\|_{\HS}^2.
\end{align*}
There are $r_1\cdots r_q$ coefficient operators, and each has $mn$
canonical entries of modulus one. Hence
\begin{equation}\label{eq:higher-mean-trace}
 \int_\Omega\tr S(x)\,d\mu(x)=mn.
\end{equation}
For every self-adjoint $S$,
\[
 \tr(S^2)
 =\|S-nI_m\|_{\HS}^2
 +2n\tr S-mn^2.
\]
Integrating this identity, using \eqref{eq:higher-mean-trace}, and substituting
in \eqref{eq:higher-average-step}, we obtain
\begin{align*}
 \|A\|^4
 &\ge\frac1m\int_\Omega\tr(S(x)^2)\,d\mu(x)\\
 &=\frac1m\int_\Omega\|S(x)-nI_m\|_{\HS}^2\,d\mu(x)
   +\frac{2n}{m}(mn)-n^2\\
 &=n^2+\frac1m\int_\Omega
 \|S(x)-nI_m\|_{\HS}^2\,d\mu(x),
\end{align*}
which is \eqref{eq:higher-order-defect}.

Assume now that $\|A\|=\sqrt n$. The nonnegative integral in
\eqref{eq:higher-order-defect} must vanish. Its integrand is continuous on the
compact space $\Omega$, and hence
\[
 \Phi_A(x)\Phi_A(x)^T=nI_m\qquad(x\in\Omega).
\]
Multilinear homogeneity gives \eqref{eq:higher-orthogonal-multiplication} for
arbitrary vectors. Conversely, if \eqref{eq:higher-orthogonal-multiplication}
holds, then for unit vectors $x^{(1)},\ldots,x^{(q)}$ the $m$ singular values
of $\Phi_A(x)$ are equal to $\sqrt n$. Equation
\eqref{eq:higher-norm-representation} then yields $\|A\|=\sqrt n$.
\end{proof}

\begin{remark}
The identity \eqref{eq:higher-orthogonal-multiplication} says that, for
$x=(x^{(1)},\ldots,x^{(q)})\in\Omega$, the operator
$n^{-1/2}\Phi_A(x)$ is a coisometry. Equivalently, the multilinear map
\[
(x^{(1)},\ldots,x^{(q)},y)
\longmapsto
\Phi_A(x^{(1)},\ldots,x^{(q)})^T y
\]
satisfies
\[
\|\Phi_A(x)^T y\|_2
=
\sqrt{n}\,
\|x^{(1)}\|_2\cdots\|x^{(q)}\|_2\|y\|_2.
\]
Thus the adjoint map is a scaled multilinear orthogonal multiplication. When $m=n$ and $x\in\Omega$, the matrix
$n^{-1/2}\Phi_A(x)$ is orthogonal. More generally, if
$x^{(1)},\ldots,x^{(q)}$ are all nonzero, then
\[
\frac{1}
{\sqrt{n}\,\prod_{\nu=1}^{q}\|x^{(\nu)}\|_2}
\Phi_A(x^{(1)},\ldots,x^{(q)})
\]
is orthogonal. Thus \Cref{thm:higher-order-defect} keeps the first $q$ variables separate
instead of flattening them into one large coordinate space; compare
\cite{LiNakatsukasaSomaUschmajew}.
\end{remark}

\begin{corollary}[A higher-order $L^2$ bound]

Assume $m=n$ and
\[
 \|A\|\le(1+\varepsilon)\sqrt n.
\]
Then
\begin{equation}
 \delta_k(A)
 =\frac1{n^3}\int_\Omega
 \left\|\Phi_A(x)\Phi_A(x)^T-nI_n\right\|_{\HS}^2\,d\mu(x)
 \le (1+\varepsilon)^4-1.
\end{equation}
\end{corollary}

\begin{proof}
This follows immediately from
\eqref{eq:higher-order-normalized-norm}.
\end{proof}

For square last variables, put
\[
 K_k(r_1,\ldots,r_q;n)
 =\frac{\Lambda_k(r_1,\ldots,r_q,n,n)}{\sqrt n}.
\]

\begin{corollary}[A higher-order lower bound]

There is a universal constant $c>0$ such that, for every $k\ge3$,
\[
 \max_{1\le\nu\le q}r_\nu
 \ge2\lceil\log_2(8n)\rceil
\]
implies
\[
 K_k(r_1,\ldots,r_q;n)-1
 \ge\frac{c}{(1+\log_2n)^4}.
\]
\end{corollary}

\begin{proof}
Fix $\nu\in\{1,\ldots,q\}$. In all first variables other than the
$\nu$th one, choose coordinate vectors. The resulting restriction of $A$ is
a trilinear sign form on
\[
 \ell_2^{r_\nu}\times\ell_2^n\times\ell_2^n,
\]
and its norm is at most $\|A\|$. Consequently,
\[
 K_k(r_1,\ldots,r_q;n)\ge K_{r_\nu,n}.
\]
Taking $\nu$ for which $r_\nu$ is maximal and applying
\Cref{thm:supercritical-gap} proves the assertion.
\end{proof}

\medskip
\begin{proof}[Proof of \Cref{thm:intro-higher-order-main}]
The map in the statement is the map defined in
\eqref{eq:higher-coefficient-map}, and the normalized estimate is precisely
\Cref{thm:higher-order-defect}. By definition, $\delta_k(A)=0$ if and only if
the nonnegative continuous integrand in
\eqref{eq:higher-normalized-defect-definition} vanishes on $\Omega$.
Multilinear homogeneity therefore shows that this condition, as well as
$\|A\|=\sqrt n$, is equivalent to
\[
 \Phi_A(x)\Phi_A(x)^T
 =n\prod_{\nu=1}^q\|x^{(\nu)}\|_2^2I_m
\]
for all choices of the first variables. The equivalent norm identity gives the
stated scaled multilinear orthogonal multiplication. The $L^2$ estimate follows
from the definition of $\delta_k(A)$ and
\[
 \delta_k(A)\le \left(\frac{\|A\|}{\sqrt n}\right)^4-1.
\]
\end{proof}

\section{Higher-order constructions}\label{sec:higher-order-constructions}

We give two further constructions from orthogonal multiplications. The first
is exact and the second is asymptotic. They are not needed in the proofs of
Main Theorems~A and~B.

\begin{theorem}[Exact tensor-product construction]\label{thm:higher-exact-construction}
Fix $k\ge3$ and put $q=k-2$. For each $\nu=1,\ldots,q$, let $N_{\nu}$ be a
Hadamard order and let $r_{\nu}\le\rho(N_{\nu})$. If
\[
N=N_{1}\cdots N_{q} \qquad\text{and}\qquad1\le m\le N,
\]
then
\[
\Lambda_{k}(r_{1},\ldots,r_{q},m,N)=\sqrt N.
\]

\end{theorem}

\begin{proof}
For each $\nu$, the square construction provides sign matrices
\[
S_{1}^{(\nu)},\ldots,S_{r_{\nu}}^{(\nu)} \in\{-1,1\}^{N_{\nu}\times N_{\nu}}
\]
such that, with
\[
S^{(\nu)}(x^{(\nu)}) =\sum_{i=1}^{r_{\nu}}x_{i}^{(\nu)}S_{i}^{(\nu)},
\]
one has
\[
S^{(\nu)}(x^{(\nu)})S^{(\nu)}(x^{(\nu)})^{T} =N_{\nu}\|x^{(\nu)}\|_{2}
^{2}I_{N_{\nu}}.
\]
Let $P_{m}:\R^{N}\to\R^{m}$ be the projection onto the first
$m$ coordinates and define
\[
B(x^{(1)},\ldots,x^{(q)}) =P_{m}\bigl(S^{(1)}(x^{(1)})\otimes\cdots\otimes
S^{(q)}(x^{(q)})\bigr).
\]
The coefficient of $x_{i_{1}}^{(1)}\cdots x_{i_{q}}^{(q)}$ is
\[
P_{m}\bigl(S_{i_{1}}^{(1)}\otimes\cdots\otimes S_{i_{q}}^{(q)}\bigr),
\]
which is an $m\times N$ sign matrix. Therefore
\[
A(x^{(1)},\ldots,x^{(q)},y,z) =\langle B(x^{(1)},\ldots,x^{(q)})z,y\rangle
\]
is unimodular.

Using $(X\otimes Y)(X\otimes Y)^{T}=(XX^{T})\otimes(YY^{T})$, we obtain
\begin{align*}
BB^{T}  &  =P_{m}\left(  \bigotimes_{\nu=1}^{q} S^{(\nu)}(x^{(\nu)})S^{(\nu
)}(x^{(\nu)})^{T}\right)  P_{m}^{T}\\
&  =N\prod_{\nu=1}^{q}\|x^{(\nu)}\|_{2}^{2}I_{m}.
\end{align*}
Hence
\[
\|B(x^{(1)},\ldots,x^{(q)})\|_{\op} =\sqrt N\prod_{\nu=1}^{q}
\|x^{(\nu)}\|_{2}.
\]
It follows that $\|A\|\le\sqrt N$. The reverse inequality follows from
\eqref{eq:universal-tensor-lower-bound}, applied to the last variable.
\end{proof}

\begin{remark}[The first quadrilinear case]
For $q=2$, $N_1=N_2=2$, the tensor-product construction gives the
quadrilinear extremizer written in coordinates in
\Cref{ex:quadrilinear-coordinate}. The two descriptions are the same after
identifying its coefficient matrices with $M_i\otimes M_j$.
\end{remark}

\begin{theorem}[Higher-order asymptotic construction]
Fix $k\ge3$ and put $q=k-2$. Let $r_{1,n},\ldots,r_{q,n}$ and $m_{n}$ be
positive integers such that $m_{n}\le n$ and
\begin{equation}
\limsup_{n\to\infty} \frac{r_{1,n}+\cdots+r_{q,n}}{\log_{2}n}<2. \label{eq:higher-subcritical-condition}
\end{equation}
Then
\[
\Lambda_{k}(r_{1,n},\ldots,r_{q,n},m_{n},n) =(1+o(1))\sqrt n.
\]

\end{theorem}

\begin{proof}
Choose $c<2$ larger than the limsup in \eqref{eq:higher-subcritical-condition}, and set
\[
a_{\nu,n}=\left\lceil \frac{r_{\nu,n}}2\right\rceil , \qquad b_{n}=\sum
_{\nu=1}^{q}a_{\nu,n}.
\]
For all sufficiently large $n$,
\[
b_{n}\le\frac c2\log_{2}n+q.
\]
Consequently,
\[
x_{n}:=\frac{n}{2^{b_{n}}} \ge2^{-q}n^{1-c/2}\longrightarrow\infty.
\]
Choose Hadamard orders $h_{n}\ge x_{n}$ with $h_{n}/x_{n}\to1$. Define
\[
N_{1,n}=2^{a_{1,n}}h_{n}, \qquad N_{\nu,n}=2^{a_{\nu,n}} \quad(2\le\nu\le q).
\]
Each $N_{\nu,n}$ is a Hadamard order and
\[
\rho(N_{\nu,n})\ge2\nu_{2}(N_{\nu,n}) \ge2a_{\nu,n}\ge r_{\nu,n}.
\]
Their product is
\[
T_{n}=\prod_{\nu=1}^{q}N_{\nu,n}=2^{b_{n}}h_{n}.
\]
Thus $T_{n}\ge n\ge m_{n}$ and
\[
\frac{T_{n}}{n}=\frac{h_{n}}{x_{n}}\longrightarrow1.
\]
\Cref{thm:higher-exact-construction} gives a form with last dimension $T_{n}$ and
norm $\sqrt{T_{n}}$. Restricting the last variable to the first $n$ coordinate
directions produces a unimodular form $A_{n}$ with
\[
\|A_{n}\|\le\sqrt{T_{n}}.
\]
By \eqref{eq:universal-tensor-lower-bound}, $\|A_{n}\|\ge\sqrt n$. Since $T_{n}/n\to1$, the conclusion follows.
\end{proof}

\begin{corollary}
Fix $k\ge3$, put $q=k-2$, and let $\alpha_{1},\ldots,\alpha_{q}\ge0$ satisfy
\[
\alpha_{1}+\cdots+\alpha_{q}<2.
\]
For every sequence $1\le m_{n}\le n$,
\[
\Lambda_{k}\bigl(
\max\{1,\left\lfloor \alpha_{1}\log_{2}n\right\rfloor \},\ldots,
\max\{1,\left\lfloor \alpha_{q}\log_{2}n\right\rfloor \},m_{n},n\bigr)
=(1+o(1))\sqrt n.
\]

\end{corollary}

\begin{remark}
If we group the first $q$ variables into one variable, their dimensions are
replaced by
\[
R_{n}=\prod_{\nu=1}^{q} r_{\nu,n}.
\]
The trilinear theorem would then require $R_n/\log_2 n<2$ asymptotically.
Our tensor-product construction only requires the additive condition \eqref{eq:higher-subcritical-condition}. For example, if $r_{\nu,n}
\sim\alpha_{\nu}\log_{2}n$ with all $\alpha_{\nu}>0$, then $R_{n}$ has order
$(\log_2 n)^{q}$, while the construction above still applies whenever
$\alpha_{1}+\cdots+\alpha_{q}<2$.
\end{remark}

\begin{remark}
The number $2$ in \eqref{eq:higher-subcritical-condition} comes from
$\rho(N)\ge2\nu_2(N)$. The proof uses a power of two with exponent about
$r_{\nu,n}/2$ for each first variable. Their total exponent must stay below
$\log_2 n$ so that the remaining Hadamard factor tends to infinity.
\end{remark}

\appendix
\section{Best rank-one approximation ratios}

This appendix gives some consequences for the spectral-to-Frobenius norm
ratio. They are not used in the main proofs. Let
\[
\mathbb{T}(d_1,\ldots,d_k)
 =\R^{d_1}\otimes\cdots\otimes\R^{d_k},
\]
with the inner product for which the canonical tensors form an orthonormal
basis. If
\[
\mathcal{X}=(x_{i_1\ldots i_k})\in\mathbb{T}(d_1,\ldots,d_k),
\]
write
\[
\|\mathcal{X}\|_2
 =\left(\sum_{i_1,\ldots,i_k}|x_{i_1\ldots i_k}|^2\right)^{1/2}.
\]
Its spectral norm is
\[
\|\mathcal{X}\|_{\sigma}
 =\max_{\|u^{(1)}\|_2=\cdots=\|u^{(k)}\|_2=1}
 \bigl|\langle\mathcal{X},
 u^{(1)}\otimes\cdots\otimes u^{(k)}\rangle\bigr|.
\]
Following Qi \cite{Qi}, the best rank-one approximation ratio is
\[
\operatorname{App}_k(\R;d_1,\ldots,d_k)
 :=\min_{\mathcal{X}\ne0}
 \frac{\|\mathcal{X}\|_{\sigma}}{\|\mathcal{X}\|_2}.
\]
This terminology comes from the identity
\[
\min_{\substack{\lambda\in\R\\
\|u^{(1)}\|_2=\cdots=\|u^{(k)}\|_2=1}}
\left\|\mathcal{X}-\lambda
u^{(1)}\otimes\cdots\otimes u^{(k)}\right\|_2^2
 =\|\mathcal{X}\|_2^2-\|\mathcal{X}\|_{\sigma}^2.
\]
Indeed, for a fixed unit rank-one tensor $\mathcal{Y}$, the best choice is
$\lambda=\langle\mathcal{X},\mathcal{Y}\rangle$.

The spectral-to-Frobenius norm ratio has been studied in several settings. Orthogonal tensors and exact attainment of the
universal lower bound are treated in \cite{LiNakatsukasaSomaUschmajew};
necessary conditions and special extremizers appear in
\cite{AgrachevKozhasovUschmajew}; probabilistic bounds are given in
\cite{KozhasovCueto}; and the nonnegative problem is studied in
\cite{CaoLiWang}. We now apply our sign constructions and equality results to this ratio.

We shall use the following theorem of Li, Nakatsukasa, Soma and Uschmajew
\cite{LiNakatsukasaSomaUschmajew}.

\begin{theorem}
\label{thm:li-app} For arbitrary positive integers $d_{1},\ldots,d_{k}$,
\[
\operatorname{App}_{k}(\R;d_{1},\ldots,d_{k}) \ge\left(  \min_{1\le
\nu\le k} \prod_{\substack{1\le\mu\le k\\\mu\ne\nu}}d_{\mu}\right)  ^{-1/2}.
\]
In particular, if $d_{1}\le\cdots\le d_{k}$, the right-hand side is
$1/\sqrt{d_{1}\cdots d_{k-1}}$. The lower bound is attained if and only if an orthogonal tensor of the
indicated format exists. Moreover, every nonzero tensor attaining equality is
a scalar multiple of an orthogonal tensor in the sense of
\cite[Theorem~3.5]{LiNakatsukasaSomaUschmajew}.
\end{theorem}

For comparison with the full tensor space, define the sign-restricted ratio
\[
\operatorname{App}_{k}^{\pm}(d_{1},\ldots,d_{k}) :=\min_{\mathcal{E}
\in\{-1,1\}^{d_{1}\times\cdots\times d_{k}}} \frac{\|\mathcal{E}\|_{\sigma}
}{\|\mathcal{E}\|_{2}}.
\]
Since every sign tensor of the indicated format satisfies
$\|\mathcal{E}\|_2=\sqrt{d_1\cdots d_k}$,
\begin{equation}
\label{eq:sign-app-lambda}\operatorname{App}_{k}^{\pm}(d_{1},\ldots,d_{k})
=\frac{\Lambda_{k}(d_{1},\ldots,d_{k})} {\sqrt{d_{1}\cdots d_{k}}}.
\end{equation}
For $k=3$, we use $\Lambda_{3}(r,m,n)=\Lambda(r,m,n)$.

The exact square theorem and its discrete gap have the following consequence.

\begin{corollary}
Let $1\le r\le n$. Then
\[
\operatorname{App}_{3}^{\pm}(r,n,n)=\frac1{\sqrt{rn}}
\]
if and only if a Hadamard matrix of order $n$ exists and $r\le\rho(n)$. If
equality fails, then
\[
\operatorname{App}_{3}^{\pm}(r,n,n) \ge\frac1{\sqrt{rn}} \left(
1+\frac2{n^{3}}\right)  ^{1/4}.
\]

\end{corollary}

\begin{proof}
By \eqref{eq:sign-app-lambda}, the first assertion is exactly
\Cref{thm:exact-square} after division by $n\sqrt r$. The second follows
from \Cref{thm:square-gap}, because
\[
\frac{(n^{2}+2/n)^{1/4}}{n\sqrt r} =\frac1{\sqrt{rn}} \left(  1+\frac2{n^{3}
}\right)  ^{1/4}.
\]

\end{proof}

The asymptotic constructions yield information not only for the restricted
class but also for the full real tensor space.

\begin{theorem}
Let $(r_{n})$ and $(m_{n})$
be positive integer sequences such that $1\le m_{n}\le n$ and
\[
\limsup_{n\to\infty}\frac{r_{n}}{\log_{2}n}<2.
\]
Then
\[
\operatorname{App}_{3}(\R;r_{n},m_{n},n) =\frac{1+o(1)}{\sqrt{r_{n}
m_{n}}}
\]
and
\[
\operatorname{App}_{3}^{\pm}(r_{n},m_{n},n) =\frac{1+o(1)}{\sqrt{r_{n} m_{n}}
}.
\]

\end{theorem}

\begin{proof}
The assumption implies $r_{n}\le n$ for all sufficiently large $n$. Hence $n$
is a largest dimension, and the symmetric form of \Cref{thm:li-app} gives
\[
\operatorname{App}_{3}(\R;r_{n},m_{n},n)\ge\frac1{\sqrt{r_{n} m_{n}}
}.
\]
Since the minimum over all real tensors is no larger than the minimum over
sign tensors,
\[
\operatorname{App}_{3}(\R;r_{n},m_{n},n) \le\operatorname{App}
_{3}^{\pm}(r_{n},m_{n},n).
\]
The trilinear asymptotic theorem and \eqref{eq:sign-app-lambda} give
\[
\operatorname{App}_{3}^{\pm}(r_{n},m_{n},n) =\frac{\Lambda(r_{n},m_{n}
,n)}{\sqrt{r_{n} m_{n} n}} =\frac{1+o(1)}{\sqrt{r_{n} m_{n}}}.
\]
The two conclusions follow by comparison.
\end{proof}

The same argument applies in every fixed order.

\begin{corollary}
Fix $k\ge3$ and put $q=k-2$. Suppose that $m_{n}\le n$
and
\[
\limsup_{n\to\infty} \frac{r_{1,n}+\cdots+r_{q,n}}{\log_{2}n}<2.
\]
Then
\[
\operatorname{App}_{k}(\R;r_{1,n},\ldots,r_{q,n},m_{n},n)
=\frac{1+o(1)} {\sqrt{r_{1,n}\cdots r_{q,n}m_{n}}}
\]
and
\[
\operatorname{App}_{k}^{\pm}(r_{1,n},\ldots,r_{q,n},m_{n},n) =\frac{1+o(1)}
{\sqrt{r_{1,n}\cdots r_{q,n}m_{n}}}.
\]

\end{corollary}

\begin{proof}
Each $r_{\nu,n}$ is at most $n$ for all sufficiently large $n$, so $n$ is a
largest dimension. Apply the symmetric form of \Cref{thm:li-app} for the
lower bound and the higher-order asymptotic theorem together with
\eqref{eq:sign-app-lambda} for the upper bound.
\end{proof}

\begin{remark}
The theorem of Li, Nakatsukasa, Soma and Uschmajew characterizes exact
equality in the unrestricted problem through orthogonal tensors. Our square
classification determines when equality can be attained by a tensor whose
canonical coefficients belong to $\{-1,1\}$. In the logarithmic range above,
the universal lower bound also has the correct leading constant, and the
asymptotically extremal tensors may be chosen to have sign coefficients.
\end{remark}

\section*{Author contributions}

Daniel M. Pellegrino proposed the problem and initiated the project. Anselmo
Raposo Junior developed the main arguments and worked on the proofs. Both
authors discussed the results, checked the mathematical details, and
contributed to the writing and revision of the manuscript.

\section*{Funding}

This work was supported by the Conselho Nacional de Desenvolvimento
Cient\'ifico e Tecnol\'ogico (CNPq, Brazil), through Grants No. 406457/2023-9
(CNPq/MCTI Call No. 10/2023) and No. 403964/2024-5 (MCTI/CNPq Call No.
16/2024). D. Pellegrino was also supported by Grant No. 305807/2025-0, and
A. Raposo Jr. by Grant No. 302341/2025-0. The funding agency had no role in
the development of the results, the preparation of the manuscript, or the
decision to submit the article.

\section*{Data availability}

No data were used for the research described in this article.

\section*{Declaration of competing interest}

The authors declare that they have no known competing financial interests or
personal relationships that could have appeared to influence the work
reported in this article.

\section*{Use of Generative AI}

OpenAI's ChatGPT was used to assist with exposition, language refinement, and
bibliographic searches. All AI-assisted material was reviewed and verified by
the authors, who take full responsibility for the content of the manuscript.


\begin{thebibliography}{99}

\bibitem {AgrachevKozhasovUschmajew}A. Agrachev, K. Kozhasov, A. Uschmajew, \emph{Chebyshev polynomials and best rank-one approximation ratio},
SIAM J. Matrix Anal. Appl. 41 (2020), no. 1, 308--331.
\href{https://doi.org/10.1137/19M1269713}{doi:10.1137/19M1269713}.

\bibitem {AlbuquerqueRezende}N.~G. Albuquerque, L.~Rezende,
\emph{Asymptotic estimates for unimodular multilinear forms with small norms
on sequence spaces}, Bull. Braz. Math. Soc. (N.S.) 52 (2021), 23--39.
\href{https://doi.org/10.1007/s00574-019-00189-2}{doi:10.1007/s00574-019-00189-2}.

\bibitem {CaoLiWang}S. Cao, S. He, Z. Li, Z. Wang, \emph{Extreme ratio between spectral and Frobenius norms of nonnegative tensors},
SIAM J. Matrix Anal. Appl. 44 (2023), no. 2, 919--944.
\href{https://doi.org/10.1137/22M1502951}{doi:10.1137/22M1502951}.

\bibitem {GeramitaPullman}A.~V. Geramita, N.~J. Pullman, \emph{A theorem of
Hurwitz and Radon and orthogonal projective modules}, Proc. Amer. Math. Soc.
42 (1974), 51--56.
\href{https://doi.org/10.1090/S0002-9939-1974-0332764-4}{doi:10.1090/S0002-9939-1974-0332764-4}.

\bibitem {GowersHatami}W.~T. Gowers, O.~Hatami,
\emph{Inverse and stability theorems for approximate representations of
finite groups}, Sb. Math. 208 (2017), no.~12, 1784--1817.
\href{https://doi.org/10.1070/SM8872}{doi:10.1070/SM8872}.

\bibitem {Horadam}K.~J. Horadam, \emph{Hadamard Matrices and Their
Applications}, Princeton University Press, Princeton, NJ, 2007.

\bibitem {Hurwitz}A.~Hurwitz, \emph{\"{U}ber die Komposition der quadratischen
Formen}, Math. Ann. 88 (1923), 1--25.
\href{https://doi.org/10.1007/BF01448439}{doi:10.1007/BF01448439}.

\bibitem {Kahane}J.-P.~Kahane, \emph{Some Random Series of Functions},
2nd ed., Cambridge Studies in Advanced Mathematics, vol.~5, Cambridge
University Press, Cambridge, 1985.

\bibitem {Knese}G.~Knese, \emph{The von Neumann inequality for $3\times 3$ matrices},
Bull. Lond. Math. Soc. 48 (2016), no.~1, 53--57.
\href{https://doi.org/10.1112/blms/bdv087}{doi:10.1112/blms/bdv087}.

\bibitem {Kosinski}Ł.~Kosiński, \emph{Three-point Nevanlinna--Pick problem in the polydisc},
Proc. Lond. Math. Soc. (3) 111 (2015), no.~4, 887--910.
\href{https://doi.org/10.1112/plms/pdv045}{doi:10.1112/plms/pdv045}.

\bibitem {KozhasovCueto}K.~Kozhasov, J.~Tonelli-Cueto,
\emph{Probabilistic bounds on best rank-1 approximation ratio}, Linear
Multilinear Algebra 72 (2024), no.~17, 3000--3028.
\href{https://doi.org/10.1080/03081087.2024.2304146}{doi:10.1080/03081087.2024.2304146}.

\bibitem {LiNakatsukasaSomaUschmajew}Z.~Li, Y.~Nakatsukasa, T.~Soma,
A.~Uschmajew, \emph{On orthogonal tensors and best rank-one approximation
ratio}, SIAM J. Matrix Anal. Appl. 39 (2018), 400--425.
\href{https://doi.org/10.1137/17M1144349}{doi:10.1137/17M1144349}.

\bibitem {ManteroTonge}A.~M. Mantero, A.~M. Tonge,
\emph{The Schur multiplication in tensor algebras}, Studia Math. 68 (1980),
no.~1, 1--24.
\href{https://doi.org/10.4064/sm-68-1-1-24}{doi:10.4064/sm-68-1-1-24}.

\bibitem {MooreRussell}C.~Moore, A.~Russell,
\emph{Approximate representations, approximate homomorphisms, and
low-dimensional embeddings of groups}, SIAM J. Discrete Math. 29 (2015),
182--197.
\href{https://doi.org/10.1137/140958578}{doi:10.1137/140958578}.

\bibitem {PellegrinoRaposo}D.~Pellegrino, A.~Raposo Jr., \emph{Constants of
the Kahane--Salem--Zygmund inequality asymptotically bounded by 1}, J. Funct.
Anal. 282 (2022), 109293.
\href{https://doi.org/10.1016/j.jfa.2021.109293}{doi:10.1016/j.jfa.2021.109293}.

\bibitem {PellegrinoSerranoSilva}D.~Pellegrino, D.~Serrano-Rodr\'iguez,
J.~Silva, \emph{On unimodular multilinear forms with small norms on sequence
spaces}, Linear Algebra Appl. 595 (2020), 24--32.
\href{https://doi.org/10.1016/j.laa.2020.02.027}{doi:10.1016/j.laa.2020.02.027}.

\bibitem {Porteous}I.~R. Porteous, \emph{Clifford Algebras and the
Classical Groups}, Cambridge Studies in Advanced Mathematics, vol.~50,
Cambridge University Press, Cambridge, 1995.

\bibitem {Qi}L.~Qi, \emph{The best rank-one approximation ratio of a tensor
space}, SIAM J. Matrix Anal. Appl. 32 (2011), 430--442.
\href{https://doi.org/10.1137/100795802}{doi:10.1137/100795802}.

\bibitem {Radon}J.~Radon, \emph{Lineare Scharen orthogonaler Matrizen}, Abh.
Math. Sem. Univ. Hamburg 1 (1922), 1--14.
\href{https://doi.org/10.1007/BF02940576}{doi:10.1007/BF02940576}.

\bibitem {RegevVidick}O.~Regev, T.~Vidick,
\emph{Bounds on dimension reduction in the nuclear norm}, in B.~Klartag and
E.~Milman (eds.), \emph{Geometric Aspects of Functional Analysis: Israel Seminar (GAFA) 2017--2019, Vol.~II},
Lecture Notes in Math., vol.~2266, Springer, Cham, 2020, 279--299.
\href{https://doi.org/10.1007/978-3-030-46762-3_13}{doi:10.1007/978-3-030-46762-3\_13}.

\bibitem {SalemZygmund}R.~Salem, A.~Zygmund, \emph{Some properties of
trigonometric series whose terms have random signs}, Acta Math. 91 (1954),
245--301.
\href{https://doi.org/10.1007/BF02393433}{doi:10.1007/BF02393433}.

\bibitem {Seberry}J.~Seberry, \emph{Orthogonal Designs: Hadamard Matrices,
Quadratic Forms and Algebras}, Springer, Cham, 2017.
\href{https://doi.org/10.1007/978-3-319-59032-5}{doi:10.1007/978-3-319-59032-5}.

\bibitem {Serre}J.-P.~Serre, \emph{Linear Representations of Finite
Groups}, Graduate Texts in Mathematics, vol.~42, Springer-Verlag, New York,
1977.

\bibitem {Shapiro}D.~B. Shapiro, \emph{Compositions of Quadratic Forms},
De Gruyter Expositions in Mathematics, vol.~33, Walter de Gruyter, Berlin--New York, 2000.
\href{https://doi.org/10.1515/9783110824834}{doi:10.1515/9783110824834}.

\bibitem {Varopoulos}N.~Th. Varopoulos,
\emph{On an inequality of von Neumann and an application of the metric theory
of tensor products to operators theory}, J. Funct. Anal. 16 (1974), 83--100.
\href{https://doi.org/10.1016/0022-1236(74)90071-8}{doi:10.1016/0022-1236(74)90071-8}.

\end{thebibliography}
\end{document}